\documentclass[11pt]{article}
\usepackage[margin=1in]{geometry}

\usepackage{graphicx} 
\usepackage{amsmath} 
\usepackage{amssymb}  
\usepackage{amsthm}

\usepackage{color}
\allowdisplaybreaks[1]

\usepackage{algorithm}
\usepackage{algpseudocode}
\usepackage{hyperref} 

\theoremstyle{plain}  
\newtheorem{theorem}{Theorem}[section]
\newtheorem{lem}[theorem]{Lemma}
\newtheorem{prop}[theorem]{Proposition}
\newtheorem{cor}[theorem]{Corollary}
\newtheorem{defn}[theorem]{Definition}

\theoremstyle{definition}

\theoremstyle{remark} 

\newtheorem{remark}{Remark}[section]

\def\0{{\bf 0}}
\def\1{{\bf 1}}

\def \M{ \mathcal{M}}
\def \bmat{\left[\begin{matrix}}
\def \emat{\end{matrix}\right]}
\def \xyvec{\left[\begin{matrix}x\\y\end{matrix}\right]}
\def \xy1vec{\left[\begin{matrix}x\\y\\1\end{matrix}\right]}
\def \QED{\begin{flushright}\Halmos\end{flushright}\end{proof}}
\def \defeq{\mathrel{\mathop{:}}=}
\def \Tr{\mathrm{Tr}}

\long\def\old#1{}

\definecolor{DarkerGreen}{RGB}{0,170,0}
\long\def\jeff#1{{\color{black}#1}}
\definecolor{orange}{rgb}{1,0.5,0}

\long\def\jeffb#1{{\color{black}#1}}
\long\def\jeffr#1{{\color{black}#1}}
\long\def\jeffo#1{{\color{black}#1}}
\long\def\jeffg#1{{\color{black}#1}}

\newenvironment{aaa}{\color{black}}

\title{\LARGE \bf Semidefinite Programming\\and Nash Equilibria in Bimatrix Games}
 \author{Amir Ali Ahmadi and Jeffrey Zhang \thanks{The authors are partially supported by the DARPA Young Faculty Award, the Young Investigator Award of the AFOSR, the CAREER Award of the NSF, the Google Faculty Award, and the Sloan Fellowship.}}
\begin{document}
\date{}
\maketitle


\begin{abstract}
\noindent We explore the power of semidefinite programming (SDP) for finding additive $\epsilon$-approximate Nash equilibria in bimatrix games. We introduce an SDP relaxation for a quadratic programming formulation of the Nash equilibrium (NE) problem and provide a number of valid inequalities to improve the quality of the relaxation. If a rank-1 solution to this SDP is found, then an exact NE can be recovered. We show that for a \jeff{strictly competitive} game, our SDP is guaranteed to return a rank-1 solution. We propose two algorithms based on iterative linearization of smooth nonconvex objective functions \jeffg{whose global minima by design coincide with rank-1 solutions}. Empirically, we {\aaa demonstrate} that these algorithms often recover solutions of rank at most two and $\epsilon$ close to zero. Furthermore, we prove that if a rank-2 solution to our SDP is found, then a $\frac{5}{11}$-NE can be recovered for any game, or a $\frac{1}{3}$-NE for a symmetric game. We then show how our SDP approach can address two (NP-hard) problems of economic interest: finding the maximum welfare achievable under any NE, and testing whether there exists a NE where a particular set of strategies is not played. \jeff{Finally, we show the connection between our SDP and the first level of the Lasserre/sum of squares hierarchy.} 
\end{abstract}

\jeffg{{\bf Keywords}: Nash equilibria, semidefinite programming, correlated equilibria.}

\section{Introduction}\label{intro}

A bimatrix game is a game between two players (referred to in this paper as players A and B) defined by a pair of $m \times n$ payoff matrices $A$ and $B$. Let $\triangle_m$ and $\triangle_n$ denote the $m$-dimensional and $n$-dimensional simplices $$\triangle_m = \{x \in \mathbb{R}^m |\  x_i \ge 0, \forall i, \sum_{i = 1}^m x_i = 1\}, \triangle_n = \{y \in \mathbb{R}^n |\  y_i \ge 0, \forall i, \sum_{i = 1}^n y_i = 1\}.$$ These form the strategy spaces of player A and player B respectively.
For a strategy pair $(x,y)\in \triangle_m \times \triangle_n $, the payoff received by player A (resp. player B) is $x^TAy$ (resp. $x^TBy$). In particular, if the players pick vertices $i$ and $j$ of their respective simplices (also called pure strategies), their payoffs will be $A_{i,j}$ and $B_{i,j}$. One of the prevailing solution concepts for bimatrix games is the notion of \emph{Nash equilibrium}. At such an equilibrium, the players are playing mutual best responses, i.e., a payoff maximizing strategy against the opposing player's strategy. In our notation, a Nash equilibrium for the game $(A,B)$ is a pair of strategies $(x^*, y^*)\in \triangle_m \times \triangle_n $ such that $$x^{*T}Ay^* \ge x^TAy^*, \forall x\in \triangle_m,$$ and $$x^{*T}By^* \ge x^{*T}By, \forall y\in \triangle_n.\footnote{In this paper we assume that all entries of $A$ and $B$ are between 0 and 1, and argue at the beginning of Section~\ref{Intro to SDP} why this is without loss of generality for the purpose of computing Nash equilibria.}$$ 

Nash~\cite{nash1951} proved that for any bimatrix game, such pairs of strategies exist (in fact his result more generally applies to games with a finite number of players and a finite number of pure strategies). While existence of these equilibria is guaranteed, finding them is believed to be a computationally intractable problem. More precisely, a result of Daskalakis, Goldberg, and Papadimitriou~\cite{daskalakis2009complexity} implies that computing Nash equilibria is PPAD-complete (see~\cite{daskalakis2009complexity} for a definition) even when the number of players is 3. This result was later improved by Chen and Deng~\cite{chen2006settling} who showed the same hardness result for bimatrix games.

These results motivate the notion of an approximate Nash equilibrium, a solution concept in which players receive payoffs ``close'' to their best response payoffs. More precisely, a pair of strategies $(x^*, y^*)\in \triangle_m \times \triangle_n$ is an \emph{(additive) $\epsilon$-Nash equilibrium} for the game $(A,B)$ if $$x^{*T}Ay^* \ge x^TAy^* - \epsilon, \forall x\in \triangle_m,$$ and $$x^{*T}By^* \ge x^{*T}By- \epsilon, \forall y\in \triangle_n. \footnote{\jeffb{There are also other important notions of approximate Nash equilibria, such as $\epsilon$-well-supported Nash equilibria~\cite{fearnley2016approximate} and relative approximate Nash equilibria~\cite{daskalakis2013complexity} which are not considered in this paper.}}$$ 
{\aaa Note that when $\epsilon=0$, $(x^*,y^*)$ form an exact Nash equilibrium, and hence it is of interest to find $\epsilon$-Nash equilibria with $\epsilon$ small. Unfortunately, approximation of Nash equilibria has also proved to be computationally difficult.} Cheng, Deng, and Teng have shown in~\cite{chen2006computing} that, unless PPAD $\subseteq$ P, there cannot be a fully polynomial-time approximation scheme for computing Nash equilibria in bimatrix games. There have, however, been a series of constant factor approximation algorithms for this problem~\cite{daskalakis2007progress, daskalakis2006note, kontogiannis2006polynomial, tsaknakis2007optimization}, with the current best producing a .3393 approximation via an algorithm by Tsaknakis and Spirakis~\cite{tsaknakis2007optimization}.

We remark that there are exponential-time algorithms for computing Nash equilibria, such as the Lemke-Howson algorithm~\cite{lemke1964equilibrium, savani2006hard}. There are also certain subclasses of the problem which can be solved in polynomial time, the most notable example being the case of zero-sum games (i.e. when $B=-A$). This problem was shown to be solvable via linear programming by Dantzig~\cite{dantzig1951proof}, and later shown to be polynomially equivalent to linear programming by Adler~\cite{adler2013equivalence}. Aside from computation of Nash equilibria, there are a number of related decision questions which are of economic interest but unfortunately NP-hard. Examples include deciding whether a player's payoff exceeds a certain threshold in some Nash equilibrium, deciding whether a game has a unique Nash equilibrium, or testing whether there exists a Nash equilibrium where a particular set of strategies is not played~\cite{gilboa1989nash, conitzer2002complexity}.


Our focus in this paper is on understanding the power of semidefinite programming\footnote{The unfamiliar reader is referred to~\cite{vandenberghe1996semidefinite} for the theory of SDPs and a description of polynomial-time algorithms for them based on interior point methods.} (SDP) for finding approximate Nash equilibria in bimatrix games or providing certificates for related decision questions. {\aaa The goal is not to develop a competitive solver, but rather to analyze the algorithmic power of SDP when applied to basic problems around computation of Nash equilibria.}
\jeffg{Semidefinite programming relaxations have been analyzed in depth in areas such as combinatorial optimization~\cite{goemans1995improved},~\cite{lovasz1979shannon} and systems theory~\cite{boyd1994linear}, but not to such an extent in game theory. To our knowledge, the appearance of SDP in the game theory literature includes the work of Stein for exchangeable equilibria in symmetric games~\cite{stein1943exchangeable}, of Parrilo on zero-sum polynomial games~\cite{parrilo2006polynomial}, of Parrilo and Shah for zero-sum stochastic games~\cite{shah2007polynomial}, and of Laraki and Lasserre for semialgebraic min-max problems in static and dynamic games~\cite{laraki2012semidefinite}.}



\subsection{Organization and Contributions of the Paper}
\jeff{In Section~\ref{Intro to SDP}, we formulate the problem of finding a Nash equilibrium in a bimatrix game as a nonconvex quadratically constrained quadratic program and pose a natural SDP relaxation for it.} In Section~\ref{Exactness Zero Sum}, we show that our SDP is exact when the game is \jeff{strictly competitive} (see Definition~\ref{Strictly Competitive Def}).
\jeffg{In Section~\ref{Algorithms}, we design two continuous but nonconvex objective functions for our SDP whose global minima coincide with rank-1 solutions. We provide a heuristic based on iterative linearization for minimizing both objective functions. We show empirically that {\aaa these approaches produce} $\epsilon$ very close to zero (on average in the order of $10^{-3}$).}
In Section~\ref{Bounds}, we {\aaa establish} a number of bounds on the quality of the approximate Nash equilibria that {\aaa can be read off of feasible solutions to our SDP}. \jeffg{In Theorems \ref{nksl}, \ref{Thm: Trace minus 2xxt}, and \ref{Thm: 1-1/k}, we show that when the SDP returns solutions which are ``close'' to rank-1, the resulting strategies have have small $\epsilon$.} We then present an improved analysis in the rank-2 case which shows how one can recover a $\frac{5}{11}$-Nash equilibrium from the SDP solution (Theorem~\ref{5/11e}). We further prove that for symmetric games (i.e., when $B = A^T$), a $\frac{1}{3}$-Nash equilibrium can be recovered in the rank-2 case (Theorem~\ref{1/3se}). We do not currently know of a polynomial-time algorithm for finding rank-2 solutions to our SDP. If such an algorithm were found, it would, together with our analysis, improve the best known approximation bound for symmetric games. 	
In Section~\ref{Bounding Payoffs and Strategy Exclusion}, we show how our SDP formulation can be used to provide certificates for certain (NP-hard) questions of economic interest about Nash equilibria \jeffg{in symmetric games}. These are the problems of testing whether the maximum welfare achievable under any \jeffg{symmetric} Nash equilibrium exceeds some threshold, and whether a set of strategies is played in every \jeffg{symmetric} Nash equilibrium. In Section~\ref{Sec: Connection to Sum of Squares/Lasserre Hierarchy}, we show that the SDP analyzed in this paper dominates the first level of the Lasserre hierarchy~(Proposition~\ref{strongerLasserre}). Some directions for future research are discussed in Section~\ref{Conclusion}. \jeffg{The four appendices of the paper add some numerical and technical details.}

\section{The Formulation of our SDP Relaxation} \label{Intro to SDP}
In this section we present an SDP relaxation for the problem of finding Nash equilibria in bimatrix games. This is done after a straightforward reformulation of the problem as a nonconvex quadratically constrained quadratic program. Throughout the paper the following notation is used. 
\begin{itemize}
	\setlength{\itemsep}{1pt}
	\setlength{\parskip}{0pt}
	\setlength{\parsep}{0pt}
	\renewcommand\labelitemi{$\cdot$}
	\item $A_{i,}$ refers to the $i$-th row of a matrix $A$.
	\item $A_{,j}$ refers to the $j$-th column of a matrix $A$.
	\item $e_i$ refers to the elementary vector $(0,\ldots, 0, 1, 0, \ldots, 0)^T$ with the 1 being in position $i$.
	\item $\triangle_k$ refers to the $k$-dimensional simplex.
	\item $1_m$ refers to the $m$-dimensional vector of one's.
	\item $0_m$ refers to the $m$-dimensional vector of zero's.
	\item $J_{m,n}$ refers to the $m \times n$ matrix of one's.
	\item $A \succeq 0$ denotes that the matrix $A$ is positive semidefinite (psd), i.e., has nonnegative eigenvalues.
	\item $A \ge 0$ denotes that the matrix $A$ is nonnegative, i.e., has nonnegative entries.
	\item $A \succeq B$ denotes that $A - B \succeq 0$.
	\item $\mathbb{S}^{k \times k}$ denotes the set of symmetric $k \times k$ matrices.
	\item $\Tr(A)$ denotes the trace of a matrix $A$, i.e., the sum of its diagonal elements.
	\item $A \otimes B$ denotes the Kronecker product of matrices $A$ and $B$.
	\item $vec(M)$ denotes the vectorized version of a matrix $M$.
	\jeffr{\item For a vector $v$, $diag(v)$ denotes the diagonal matrix with $v$ on its diagonal. For a square matrix $M$, $diag(M)$ denotes the vector containing its diagonal entries.}
\end{itemize}

We also assume that all entries of the payoff matrices $A$ and $B$ are between 0 and 1. This can be done without loss of generality because Nash equilibria are invariant under certain affine transformations in the payoffs. In particular, the games $(A,B)$ and ${(cA+dJ_{m\times n}, eB+fJ_{m \times n})}$ have the same Nash equilibria for any scalars $c,d,e,$ and $f$, with $c$ and $e$ positive. This is because
\begin{equation*}
\begin{aligned}
x^{*T}Ay &\ge x^TAy\\
\Leftrightarrow c(x^{*T}Ay^*) + d &\ge c(x^TAy^*)+d\\
\Leftrightarrow c(x^{*T}Ay^*) + d(x^{*T}J_{m\times n}y^*) &\ge c(x^TAy^*)+d(x^TJ_{m\times n}y^*)\\
\Leftrightarrow x^{*T}(cA+dJ_{m\times n})y^* &\ge x^T(cA+dJ_{m\times n})y
\end{aligned}
\end{equation*}
Identical reasoning applies for player B. 

\subsection{Nash Equilibria as Solutions to Quadratic Programs}\label{Nash as QP}
Recall the definition of a Nash equilibrium from Section~\ref{intro}. \jeff{An equivalent characterizaiton is that a strategy pair $(x^*, y^*)\in \triangle_m \times \triangle_n$ is a Nash equilibrium for the game $(A,B)$ if and only if
\begin{equation}\label{QP}
\begin{aligned}
x^{*T}Ay^* \ge e_i^TAy^*, \forall i \in \{1,\ldots, m\},\\
x^{*T}By^* \ge x^{*T}Be_i, \forall i \in \{1,\ldots, n\}.
\end{aligned}
\end{equation} 

The equivalence can be seen by noting that because the payoff from playing any mixed strategy is a convex combination of payoffs from playing pure strategies, there is always a pure strategy best response to the other player's strategy.}


We now treat the Nash problem as the following quadratic programming (QP) feasibility problem:\\
\begin{equation}\label{QCQP Formulation}
\begin{aligned}
& \underset{x \in \mathbb{R}^m, y \in \mathbb{R}^n}{\min}
& & 0 \\
& \text{subject to}
& & x^TAy \ge e_i^TAy, \forall i\in \{1,\ldots, m\},\\
&&& x^TBy \ge x^TBe_j, \forall j\in \{1,\ldots, n\},\\
&&& x_i \ge 0, \forall i\in \{1,\ldots, m\},\\
&&& y_i \ge 0, \forall j\in \{1,\ldots, n\},\\
&&& \sum_{i=1}^m x_i = 1,\\
&&& \sum_{i=1}^n y_i = 1.\\
\end{aligned}
\end{equation}

\jeff{Similarly, a pair of strategies} $x^* \in \triangle_m$ and $y^* \in \triangle_n$ form an $\epsilon$-Nash equilibrium for the game $(A,B)$ if and only if
$$x^{*T}Ay^* \ge e_i^TAy^*-\epsilon, \forall i \in \{1,\ldots, m\},$$
$$x^{*T}By^* \ge x^{*T}Be_i-\epsilon, \forall i \in \{1,\ldots, n\}.$$
Observe that any pair of simplex vectors $(x, y)$ is an $\epsilon$-Nash equilibrium for the game $(A,B)$ for any $\epsilon$ that satisfies
$$\epsilon \ge \max\{\underset{i}{\max}\ e_i^TAy - x^TAy, \underset{i}{\max}\ x^TBe_i - x^TBy\}.$$


We use the following notation throughout the paper:
\begin{itemize}
	\setlength{\itemsep}{1pt}
	\setlength{\parskip}{0pt}
	\setlength{\parsep}{0pt}
	\renewcommand\labelitemi{$\cdot$}
	\item $\epsilon_A(x,y) \mathrel{\mathop:}= \underset{i}{\max}\ e_i^TAy - x^TAy$,
	\item $\epsilon_B(x,y) \mathrel{\mathop:}= \underset{i}{\max}\ x^TBe_i - x^TBy$,
	\item $\epsilon(x,y) \mathrel{\mathop:}= \max\{\epsilon_A(x,y), \epsilon_B(x,y)\}$,
\end{itemize}
and the function parameters are later omitted if they are clear from the context.

\subsection{SDP Relaxation}\label{SDP Relaxation}

The QP formulation in (\ref{QCQP Formulation}) lends itself to a natural SDP relaxation. We define a matrix
$$\mathcal{M} \mathrel{\mathop:}= \left[\begin{matrix} X & P\\ Z & Y\end{matrix}\right],$$
and an augmented matrix
$$\mathcal{M}' \mathrel{\mathop:}= \left[\begin{matrix} X & P & x\\ Z & Y & y\\ x & y & 1\end{matrix}\right],$$
with $X \in S^{m\times m}, Z \in \mathbb{R}^{n\times m}, Y \in S^{n\times n}, x \in \mathbb{R}^m, y \in \mathbb{R}^n$ and $P = Z^T$.\\\\
The SDP relaxation can then be expressed as

\begin{align}\label{SDP1}\tag{SDP1}
& \underset{\mathcal{M}' \in \mathbb{S}^{m+n+1, m+n+1}}{\min}
& & 0 \nonumber\\
& \text{subject to}
& & \Tr(AZ) \ge e_i^TAy, \forall i\in \{1,\ldots, m\}, \label{SDP1 Relaxed Nash A}\\
&&& \Tr(BZ) \ge x^TBe_j, \forall j\in \{1,\ldots, n\},\label{SDP1 Relaxed Nash B}\\
&&& \sum_{i=1}^m x_i = 1, \label{SDP1 Unity x}\\
&&& \sum_{i=1}^n y_i = 1, \label{SDP1 Unity y}\\
&&& \jeffg{\M' \ge 0}, \label{SDP1 Nonnegativity}\\
&&& \M'_{m+n+1,m+n+1}=1,\\
&&& \mathcal{M}' \label{SDP1 PSD}\succeq 0.
\end{align}

We refer to the constraints (\ref{SDP1 Relaxed Nash A}) and (\ref{SDP1 Relaxed Nash B}) as the relaxed Nash constraints and the constraints (\ref{SDP1 Unity x}) and (\ref{SDP1 Unity y}) as the unity constraints. This SDP is motivated by the following observation.
\begin{prop} Let $\M'$ be any rank-1 feasible solution to~\ref{SDP1}. Then the vectors $x$ and $y$ from its last column constitute a Nash equilibrium for the game $(A,B)$.\end{prop}
\begin{proof} We know that $x$ and $y$ are in the simplex from the constraints (\ref{SDP1 Unity x}), (\ref{SDP1 Unity y}), and (\ref{SDP1 Nonnegativity}).\\
If the matrix $\mathcal{M}'$ is rank-1, then it takes the form \begin{equation}\label{Rank 1 form}
\bmat xx^T & xy^T & x \\ yx^T & yy^T & y \\x^T & y^T & 1 \emat
=\xy1vec \xy1vec^T.\end{equation} Then, from the relaxed Nash constraints we have that\\
$$e_i^TAy \le \Tr(AZ) = \Tr(Ayx^T) = \Tr(x^TAy) = x^TAy,$$
$$x^TAe_i \le \Tr(BZ) = \Tr(Byx^T) = \Tr(x^TBy) = x^TBy.$$
The claim now follows from the \jeff{characterization given in (\ref{QP})}.
\end{proof}

\begin{remark}
	Because a Nash equilibrium always exists, there will always be a matrix of the form~(\ref{Rank 1 form}) which is feasible to~\ref{SDP1}. Thus we can disregard any concerns about \ref{SDP1} being feasible, even when we add valid inequalities to it in Section~\ref{Valid Inequalities}.
\end{remark}

\begin{remark}
	It is intuitive to note that the submatrix $P=Z^T$ of the matrix $\M'$ corresponds to a probability distribution over the strategies, and that seeking a rank-1 solution to our SDP can be interpreted as making $P$ a product distribution.
\end{remark}

The following theorem shows that~\ref{SDP1} is a weak relaxation and stresses the necessity of additional valid constraints.

\begin{theorem}\label{Necessity of VI}
	Consider a bimatrix game with payoff matrices bounded in $[0,1]$. Then for any two vectors $x \in \triangle_m$ and $y \in \triangle_n$, there exists a feasible solution $\M'$ to~\ref{SDP1} with $\xy1vec$ as its last column.
\end{theorem}
\begin{proof}
Consider any $x,y,\gamma > 0,$ and the matrix
$$\xy1vec\xy1vec^T + \bmat \gamma J_{m+n,m+n} & 0_{m+n} \\ 0_{m+n}^T & 0\emat.$$ This matrix is the sum of two nonnegative psd matrices and is hence nonnegative and psd. By assumption $x$ and $y$ are in the simplex, and so constraints $(\ref{SDP1 Unity x})-(\ref{SDP1 PSD})$ of~\ref{SDP1} are satisfied. To check that constraints $(\ref{SDP1 Relaxed Nash A})$ and $(\ref{SDP1 Relaxed Nash B})$ hold, note that since $A$ and $B$ are nonnegative, as long as the matrices $A$ and $B$ are not the zero matrices, the quantities $\Tr(AZ)$ and $\Tr(BZ)$ will become arbitrarily large as $\gamma$ increases. Since $e_i^TAy$ and $x^TBe_i$ are bounded by 1 by assumption, we will have that constraints $(\ref{SDP1 Relaxed Nash A})$ and $(\ref{SDP1 Relaxed Nash B})$ hold for $\gamma$ large enough. In the case where $A$ or $B$ is the zero matrix, the Nash constraints are trivially satisfied for the respective player.
\end{proof}

\subsection{Valid Inequalities}\label{Valid Inequalities}

In this subsection, we introduce a number of valid inequalities to improve upon the SDP relaxation in~\ref{SDP1}. These inequalities are justified by being valid if the matrix returned by the SDP is rank-1. The terminology we introduce here to refer to these constraints is used throughout the paper. \jeff{Constraints (\ref{R_X}) and (\ref{R_Y}) will be referred to as the \emph{row inequalities}, and (\ref{CE_A}) and (\ref{CE_B}) will be referred to as the \emph{correlated equilibrium inequalities}.}

\jeff{

\begin{prop}\label{Prop: Valid Inequalities}
	Any rank-1 solution $\M'$ to~\ref{SDP1} must satisfy the following:

	\begin{equation}\label{R_X}
	\sum_{j=1}^m X_{i,j} = \sum_{j=1}^n  P_{i,j} = x_i, \forall i \in \{1,\ldots, m\},\end{equation}
	\begin{equation}\label{R_Y}
	\sum_{j=1}^n Y_{i,j} = \sum_{j=1}^m Z_{i,j} = y_i, \forall i \in \{1,\ldots, n\}.\end{equation}
	\begin{equation}\label{CE_A}
	\sum_{j=1}^n A_{i,j}P_{i,j} \ge \sum_{j=1}^n A_{k,j}P_{i,j}, \forall i,k \in \{1,\ldots, m\},\end{equation}
	\begin{equation}\label{CE_B}
	\sum_{j=1}^m B_{j,i}P_{j,i} \ge \sum_{j=1}^m B_{j,k}P_{j,i}, \forall i,k \in \{1,\ldots, n\}.\end{equation}

\end{prop}

\begin{proof}
	Recall from (\ref{Rank 1 form}) that if $\M'$ is rank-1, it is of the form $$\left[\begin{matrix}xx^T & xy^T & x \\ yx^T & yy^T & y \\x^T & y^T & 1\end{matrix}\right]	=\xy1vec\xy1vec^T.$$ To show (\ref{R_X}), observe that
	$$\sum_{j=1}^m X_{i,j} = \sum_{j=1}^m x_ix_j = x_i, \forall i \in \{1,\ldots, m\}.$$
	An identical argument works for the remaining matrices $P,Z,$ and $Y$. To show (\ref{CE_A}) and (\ref{CE_B}), observe that a pair $(x, y)$ is a Nash equilibrium if and only if
		$$\forall i, x_i > 0 \Rightarrow e_i^TAy = x^TAy = \underset{i}{\max}\ e_i^TAy,$$
		$$\forall i, y_i > 0 \Rightarrow x^TBe_i = x^TBy = \underset{i}{\max}\ x^TBe_i.$$
	This is because the Nash conditions require that $x^TAy$, a convex combination of $e_i^TAy$, be at least $e_i^TAy$ for all $i$. Indeed, if $x_i > 0$ but $e_i^TAy < x^TAy$, the convex combination must be less than $\underset{i}{\max}\ x^TAy$.
	
	For each $i$ such that $x_i = 0$ or $y_i = 0$, inequalities (\ref{CE_A}) and (\ref{CE_B}) reduce to $0 \ge 0$, so we only need to consider strategies played with positive probability. Observe that if $\M'$ is rank-1, then\\
	$$\sum_{j =1}^n A_{i,j}P_{i,j} = x_i\sum_{j =1}^n A_{i,j}y_j = x_ie_i^TAy \ge x_ie_k^TAy = \sum_{j =1}^nA_{k,j}P_{i,j}, \forall i,k$$
	$$\sum_{j =1}^m B_{j,i}P_{j,i} = y_i\sum_{j =1}^mB_{j,i}x_j = y_i x^TBe_i \ge y_ix^TBe_k = \sum_{j =1}^mB_{j,i}P_{j,k}, \forall i,k.$$
\end{proof}

}

\begin{remark}There are two ways to interpret the inequalities in (\ref{CE_A}) and (\ref{CE_B}): the first is as a relaxation of the constraint $x_i(e_i^TAy - e_j^TAy) \ge 0, \forall i,j$, which must hold since any strategy played with positive probability must give the best response payoff. The other interpretation is to have the distribution over outcomes defined by $P$ be a correlated equilibrium \cite{aumann1974subjectivity}. This can be imposed by a set of linear constraints on the entries of $P$ as explained next.

Suppose the players have access to a public randomization device which prescribes a pure strategy to each of them (unknown to the other player). The distribution over the assignments can be given by a matrix $P$, where $P_{i,j}$ is the probability that strategy $i$ is assigned to player A and strategy $j$ is assigned to player B. This distribution is a correlated equilibrium if both players have no incentive to deviate from the strategy prescribed, that is, if the prescribed pure strategies $a$ and $b$ satisfy
$$\sum_{j=1}^n A_{i,j}Prob(b=j|a=i) \ge \sum_{j=1}^n A_{k,j}Prob(b=j|a=i),$$
$$\sum_{i=1}^m B_{i,j}Prob(a=i|b=j) \ge \sum_{i=1}^m B_{i,k}Prob(a=i|b=j).$$

If we interpret the $P$ submatrix in our SDP as the distribution over the assignments by the public device, then because of our row constraints, $Prob(b=j|a=i)=\frac{P_{i,j}}{x_i}$ whenever $x_i \ne 0$ (otherwise the above inequalities are trivial). Similarly, $P(a=i|b=j)=\frac{P_{i,j}}{y_j}$ for nonzero $y_j$. Observe now that the above two inequalities imply (\ref{CE_A}) and (\ref{CE_B}). Finally, note that every Nash equilibrium generates a correlated equilibrium, since if $P$ is a product distribution given by $xy^T$, then $Prob(b=j|a=i)=y_j$ and $P(a=i|b=j)=x_i$.
\end{remark}

\subsubsection{Implied Inequalities} \label{Implied Inequalities}
In addition to those explicitly mentioned in the previous section, there are other natural valid inequalities which are omitted because they are implied by the ones we have already proposed. We give two examples of such inequalities in the next proposition. We refer to the constraints in~(\ref{Eq: Distribution Inequalities}) below as the \emph{distribution constraints}. The constraints in~(\ref{Eq: McCormick}) are the familiar McCormick inequalities \cite{mccormick1976computability} for box-constrained quadratic programming.

\begin{prop}
	Let $z \defeq \bmat x \\ y \emat$. Any rank-1 solution $\M'$ to~\ref{SDP1} must satisfy the following:
	\begin{equation}\label{Eq: Distribution Inequalities}
	\sum_{i = 1}^m \sum_{j =1}^m X_{i,j} = \sum_{i =1}^n \sum_{j =1}^m Z_{i,j} = \sum_{i =1}^n \sum_{j =1}^n Y_{i,j} = 1.
	\end{equation}
	
	\begin{equation}\label{Eq: McCormick}
	\begin{aligned}
	\M_{i,j} &\le z_i, \forall i,j \in \{1,\ldots,m+n\},\\
	\M_{i,j} + 1 &\ge z_{i}+z_{j}, \forall i,j \in \{1,\ldots,m+n\}.
	\end{aligned}
	\end{equation}
	
\end{prop}

\begin{proof}
	The distribution constraints follow immediately from the row constraints (\ref{R_X}) and (\ref{R_Y}), along with the unity constraints (\ref{SDP1 Unity x}) and (\ref{SDP1 Unity y}).
	
	The first McCormick inequality is immediate as a consequence of (\ref{R_X}) and (\ref{R_Y}), as all entries of $\M$ are nonnegative. To see why the second inequality holds, consider whichever submatrix $X, Y, P$, or $Z$ that contains $\M_{i,j}$. Suppose that this submatrix is, e.g., $P$. Then, since $P$ is nonnegative,
	$$0 \le \sum_{k=1, k \ne i}^{m}\sum_{l = 1, l \ne j}^n P_{k,l} \overset{(\ref{R_X})}{=} \sum_{k=1, k \ne i}^m (x_k - P_{k,j})  \overset{(\ref{R_Y})}{=} (1-x_i) - (y_j - P_{i,j}) = P_{i,j}+1-x_i-y_j.$$
	
	The same argument holds for the other submatrices, and this concludes the proof.
\end{proof}

\jeffr{
\subsection{Simplifying our SDP}\label{SSec: Simplification}

We observe that the row constraints (\ref{R_X}) and (\ref{R_Y}) along with the correlated equilibrium constraints (\ref{CE_A}) and (\ref{CE_B}) imply the relaxed Nash constraints (\ref{SDP1 Relaxed Nash A}) and (\ref{SDP1 Relaxed Nash B}). Indeed, if we fix an index $k~\in~\{1, \ldots, m\}$, then
$$\Tr(AZ) = \sum_{i=1}^m \sum_{j=1}^n A_{i,j} P_{i,j} \overset{(\ref{CE_A})}{\ge} \sum_{i=1}^m \sum_{j=1}^n A_{k,j} P_{i,j} \ge \sum_{j=1}^n A_{k,j} (\sum_{i=1}^m P_{i,j}) \overset{(\ref{R_Y}), P = Z^T}{\ge} \sum_{j=1}^n A_{k, j} y_j = e_k^TAy.$$
The proof for player B proceeds identically. Then, after collecting the valid inequalities and removing the relaxed Nash constraints, we arrive at an SDP given by
\begin{samepage}
	\begin{align}\label{SDP1.1}\tag{SDP1'}
	& \underset{\M' \in S^{(m+n+1)\times(m+n+1)}}{\min}
	& & 0 \\
	& \text{subject to}
	& & (\ref{SDP1 Unity x})-(\ref{SDP1 PSD}), (\ref{R_X})-(\ref{CE_B}).\nonumber
	\end{align}
\end{samepage}

\jeffg{We make the observation that the last row and column of $\M'$ can be removed from this SDP, that is, there is a one-to-one correspondence between solutions to \ref{SDP1.1} and those to the following SDP (where $\mathcal{M} \mathrel{\mathop:}= \left[\begin{matrix} X & P\\ Z & Y\end{matrix}\right],$ with $P = Z^T$):}

\jeffr{
\begin{samepage}
	\begin{align}\label{SDP2}\tag{SDP2}
	& \underset{\M \in S^{(m+n)\times(m+n)}}{\min}
	& & 0 \\
	& \text{subject to}
	& & \M \succeq 0, &\label{SDP2 PSD}\\
	&&& \M \ge 0, &\label{SDP2 Nonnegativity}\\
	&&& \sum_{i=1}^n \sum_{j=1}^n P_{i,j} = 1, &\label{SDP2 Distribution}\\
	&&& \sum_{j=1}^m X_{i,j} = \sum_{j=1}^n  P_{i,j}, \forall i \in \{1,\ldots, m\}, & \label{SDP2 Row x}\\
	&&& \sum_{j=1}^n Y_{i,j} = \sum_{j=1}^m Z_{i,j}, \forall i \in \{1,\ldots, n\}, & \label{SDP2 Row y}\\
	&&& \sum_{j=1}^n A_{i,j}P_{i,j} \ge \sum_{j=1}^n A_{k,j}P_{i,j}, \forall i,k \in \{1,\ldots, m\}, & \label{SDP2 CE A}\\
	&&& \sum_{j=1}^m B_{j,i}P_{j,i} \ge \sum_{j=1}^m B_{j,k}P_{j,i}, \forall i,k \in \{1,\ldots, n\}. & \label{SDP2 CE B}
	\end{align}
\end{samepage}

Indeed, it is readily verified that the submatrix $\M$ from any feasible solution $\M'$ to \ref{SDP1.1} is feasible to \ref{SDP2}. Conversely, let $\M$ be any feasible matrix to \ref{SDP2}. Consider an eigendecomposition $\M = \sum_{i=1}^k \lambda_i v_i v_i^T$ and let $\bmat x\\y\emat \defeq \M \frac{1_{m+n}}{2}.$ Then the matrix
\begin{equation}\label{Eq: M psd M' psd} \M' \defeq \bmat \M & \bmat x \\ y\emat \\ \bmat x^T & y^T \emat & 1 \emat = \sum_{i=1}^k \lambda_i \bmat v_i\\ 1_{m+n}^Tv_i/2 \emat \bmat v_i\\ 1_{m+n}^Tv_i/2 \emat^T\end{equation}}
is easily seen to be feasible to \ref{SDP1.1}.}

\jeffg{Given any feasible solution $\M$ to \ref{SDP2}, observe that the submatrix $P$ is a correlated equilibrium. We take our candidate approximate Nash equilibrium to be the pair $x = P1_n$ and $y = P^T1_m$. If the correlated equilibrium $P$ is rank-1, then the pair $(x,y)$ so defined constitutes an exact Nash equilibrium. In Section~\ref{Algorithms}, we will add certain objective functions to \ref{SDP2} with the interpretation of searching for low-rank correlated equilibria.}

\section{Exactness for \jeff{Strictly Competitive} Games}\label{Exactness Zero Sum}

\jeff{In this section, we show that \ref{SDP1} recovers a Nash equilibrium for any zero-sum game, and that \ref{SDP2} recovers a Nash equilibrium for any strictly competitive game (see Definition \ref{Strictly Competitive Def} below). Both these notions represent games where the two players are in direct competition, but strictly competitive games are more general, and for example, allow both players to have nonnegative payoff matrices. These classes of games are solvable in polynomial time via linear programming. Nonetheless, it is reassuring to know that our SDPs recover these important special cases.}
\begin{defn}\label{Zero Sum Def}
	A \emph{zero-sum game} is a game in which the payoff matrices satisfy $A = -B$.
\end{defn}

\begin{theorem}\label{Zero Sum}
	For a zero-sum game, the vectors $x$ and $y$ from the last column of any feasible solution $\M'$ to~\ref{SDP1} constitute a Nash equilibrium.
\end{theorem}
\begin{proof}
	Recall that the relaxed Nash constraints (\ref{SDP1 Relaxed Nash A}) and (\ref{SDP1 Relaxed Nash B}) read
	$$\Tr(AZ) \ge e_i^TAy, \forall i \in \{1,\ldots, m\},$$
	$$\Tr(BZ) \ge x^TBe_j, \forall j \in \{1,\ldots, n\}.$$
	Since $B=-A$, the latter statement is equivalent to
	$$\Tr(AZ) \le x^TAe_j, \forall j \in \{1,\ldots, n\}.$$
	In conjunction these imply
	\begin{equation}\label{eq: zero sum inequality}
	e_i^TAy \le \Tr(AZ) \le x^TAe_j, \forall i \in \{1,\ldots, m\}, j \in \{1,\ldots, n\}.
	\end{equation}
	
	We claim that any pair $x \in \triangle_m$ and $y \in \triangle_n$ which satisfies the above condition is a Nash equilibrium. To see that $x^TAy \ge e_i^TAy, \forall i \in \{1,\ldots, m\},$ observe that $x^TAy$ is a convex combination of $x^TAe_j$, which are at least $e_i^TAy$ by (\ref{eq: zero sum inequality}). To see that $x^TBy \ge x^TBe_j \Leftrightarrow x^TAy \le x^TAe_j, \forall j \in \{1,\ldots, n\}$, observe that $x^TAy$ is a convex combination of $e_i^TAy$, which are at most $x^TAe_j$ by (\ref{eq: zero sum inequality}).
	
\end{proof}

\jeff{
\begin{defn}\label{Strictly Competitive Def}
	A game $(A, B)$ is \emph{strictly competitive} if for all $x, x' \in \triangle_m, y, y' \in \triangle_n$, $x^TAy - x'^TAy'$ and $x'^TBy' - x^TBy$ have the same sign.
\end{defn}

The interpretation of this definition is that if one player benefits from changing from one outcome to another, the other player must suffer. Adler, Daskalakis, and Papadimitriou show in \cite{adler2009note} that the following much simpler characterization is equivalent.

\begin{theorem}[Theorem 1 of \cite{adler2009note}]\label{Thm: Adler}
	A game is strictly competitive if and only if there exist scalars $c,d,e$, and $f,$ with $c > 0, e > 0,$ such that $cA+dJ_{m \times n} = -eB + fJ_{m \times n}$.
\end{theorem}}

One can easily show that there exist \jeff{strictly competitive games} for which not all feasible solutions to~\ref{SDP1} have Nash equilibria as their last columns (see Theorem~\ref{Necessity of VI}). However, we show that this is the case for~\ref{SDP2}.

\jeff{\begin{theorem}\label{Strictly Competitive Exact}
	For a \jeff{strictly competitive game}, \jeffo{the vectors $x \defeq P1_n$ and $y \defeq P^T1_m$ from any feasible solution $\M$ to~\ref{SDP2} constitute a Nash equilibrium.}
\end{theorem}

To prove Theorem \ref{Strictly Competitive Exact} we need the following lemma, which shows that feasibility of a matrix \jeffg{$\M$} in~\ref{SDP2} is invariant under certain transformations of $A$ and $B$.
}

\begin{lem}\label{scaling shifting}
	Let $c,d,e$, and $f$ be any set of scalars with $c > 0$ and $e > 0$. If a matrix \jeffg{$\M$} is feasible to~\ref{SDP2} with input payoff matrices $A$ and $B$, then it is also feasible to~\ref{SDP2} with input matrices $cA+dJ_{m \times n}$ and $eB + fJ_{m \times n}$.
\end{lem}
\begin{proof}

\jeffr{
It suffices to check that constraints (\ref{SDP2 CE A}) and (\ref{SDP2 CE B}) of~\ref{SDP2} still hold, as only the correlated equilibrium constraints use the matrices $A$ and $B$. We only show that constraint (\ref{SDP2 CE A}) still holds because the argument for constraint (\ref{SDP2 CE B}) is identical.

Note from the definition of $x$ that for each $i \in \{1,\ldots, m\}, x_i = \sum_{j=1}^n (J_{m\times n})_{i,j}P_{i,j}$. To check that the correlated equilibrium constraints hold, observe that for scalars $c>0,d$, and for all $i,k \in \{1,\ldots, m\}$,
\begin{equation*}
\begin{aligned}
\sum_{j=1}^n A_{i,j}P_{i,j} &\ge \sum_{j=1}^n A_{k,j}P_{i,j}\\
\Leftrightarrow c\sum_{j=1}^n A_{i,j}P_{i,j} + d \sum_{j=1}^n P_{i,j} &\ge c\sum_{j=1}^n A_{k,j}P_{i,j} + d \sum_{j=1}^n P_{i,j}\\
\Leftrightarrow c\sum_{j=1}^n A_{i,j}P_{i,j} + d\sum_{j=1}^n (J_{m \times n})_{i,j}P_{i,j} &\ge c\sum_{j=1}^n A_{k,j}P_{i,j} + d\sum_{j=1}^n (J_{m \times n})_{k,j}P_{i,j}\\
\Leftrightarrow \sum_{j=1}^n (cA_{i,j}+dJ_{m \times n})_{k,j}P_{i,j} &\ge \sum_{j=1}^n (cA_{i,j}+dJ_{m \times n})_{k,j}P_{i,j}.
\end{aligned}
\end{equation*}
}
\end{proof}

\begin{proof}[\jeff{Proof (of Theorem \ref{Strictly Competitive Exact})}]
	Let $A$ and $B$ be the payoff matrices of the given \jeff{strictly competitive game} and let \jeffo{$\M$} be a feasible solution to~\ref{SDP2}. Since the game is \jeff{strictly competitive}, we know from Theorem \ref{Thm: Adler} that $cA + dJ_{m\times n} = -eB + fJ_{m\times n}$ for some scalars $c>0,e>0,d, f$. Consider a new game with input matrices $\tilde{A} = cA+dJ_{m\times n}$ and $\tilde{B} = eB - fJ_{m\times n}$. By Lemma~\ref{scaling shifting}, $\M$ is still feasible to~\ref{SDP2} with input matrices $\tilde{A}$ and $\tilde{B}$. \jeffg{By the arguments in Section \ref{SSec: Simplification}, the matrix $\M' \defeq \bmat \M & \bmat x \\ y\emat \\ \bmat x^T & y^T \emat & 1 \emat$ is feasible to \ref{SDP1.1}, and hence also to~\ref{SDP1}.} Now notice that since $\tilde{A} = -\tilde{B}$, Theorem~\ref{Zero Sum} implies that the vectors $x$ and $y$ in the last column form a Nash equilibrium to the game $(\tilde{A}, \tilde{B})$. Finally recall from the arguments at the beginning of Section~\ref{Intro to SDP} that Nash equilibria are invariant to scaling and shifting of the payoff matrices, and hence $(x, y)$ is a Nash equilibrium to the game $(A, B)$.
\end{proof}


\section{Algorithms for Lowering Rank}\label{Algorithms}

In this section, we present heuristics which aim to find low-rank solutions to~\ref{SDP2} and present some empirical results. Recall that our~\ref{SDP2} in Section~\ref{SSec: Simplification} did not have an objective function. Hence, we can encourage low-rank solutions by choosing certain objective functions, in particular the trace of the matrix $\M$, which is a general heuristic for minimizing the rank of symmetric matrices~\cite{recht2010guaranteed,fazel2002matrix}. 
This simple objective function is already guaranteed to produce a rank-1 solution in the case of \jeff{strictly competitive games} (see Proposition~\ref{Zero Sum Rank 1} below). 
For general games, however, one can design better objective functions in an iterative fashion (see Section~\ref{Linearization Algorithms}).

\jeffg{\emph{Notational Remark:} For the remainder of this section, we will use the shorthand $x \defeq P1_n$ and $y \defeq P^T1_m$, where $P$ is the upper right submatrix of a feasible solution $\M$ to \ref{SDP2}.}

\begin{prop}\label{Zero Sum Rank 1} For a \jeff{strictly competitive game}, any optimal solution to~\ref{SDP2} with $\Tr(\M)$ as the objective function must be rank-1.\end{prop}

\begin{proof}
	
	\jeffr{Let $$\M \mathrel{\mathop:}= \bmat X & P\\ P^T & Y\emat$$ be a feasible solution to~\ref{SDP2}.} In the case of \jeff{strictly competitive games}, from Theorem~\ref{Strictly Competitive Exact} we know that that $(x, y)$ is a Nash equilibrium. Then because the matrix $\M$ is psd, \jeffg{from (\ref{Eq: M psd M' psd}) and an application of the Schur complement \jeff{(see, e.g. \cite[Sect. A.5.5]{boyd2004convex})} to $\bmat \M & \bmat x \\ y\emat \\ \bmat x^T & y^T \emat & 1 \emat$, we have that $\M \succeq \xyvec\xyvec^T$. Hence, $\M = \bmat xx^T & xy^T \\ yx^T & yy^T\emat+ \mathcal{P}$ for some psd matrix $\mathcal{P}$ and \jeffg{the} Nash equilibrium $(x, y)$.} Given this expression, the objective \jeffg{function $\Tr(\M)$} is then $x^Tx + y^Ty + \Tr(\mathcal{P})$. As $(x,y)$ is a Nash equilibrium, the choice of $\mathcal{P} = 0$ results in a feasible solution. Since the zero matrix has the minimum possible trace among all psd matrices, the solution will be the rank-1 matrix $\xyvec\xyvec^T$.\end{proof}

\jeffg{\begin{remark}
If the row constraints and the nonnegativity constraints on $X$ and $Y$ are removed from \ref{SDP2}, then this SDP with $\Tr(\M)$ as the objective function can be interpreted as searching for a minimum-rank correlated equilibrium $P$ via the nuclear norm relaxation; see~\cite[Section 2]{recht2010guaranteed}.
\end{remark}}

\subsection{Linearization Algorithms}\label{Linearization Algorithms}
The algorithms we present in this section \jeffg{for minimzing the rank of the matrix $\M$ in \ref{SDP2}} are based on iterative linearization of certain nonconvex objective functions. Motivated by the next proposition, we design two continuous (nonconvex) objective functions that, if minimized exactly, would guarantee rank-1 solutions. We will then linearize these functions iteratively.


\begin{prop}\label{diagonal sufficiency} Let the matrices $X$ and $Y$ and vectors \jeffr{$x \defeq P1_n$ and $y \defeq P^T1_m$} be taken from a feasible solution \jeffr{}to~\ref{SDP2}. Then the matrix $\M$ is rank-1 if and only if $X_{i,i} = x_i^2$ and $Y_{i,i} = y_i^2$ for all $i$.\end{prop}
\begin{proof}
\jeffr{Note that if $\M$ is rank-1, then it can be written as $zz^T$ for some $z \in \mathbb{R}^{m+n}$. The $i$-th diagonal entry in the $X$ submatrix will then be equal to $$z_i^2 \overset{(\ref{Eq: Distribution Inequalities})}{=} \frac{1}{4}z_i^2(1_{m+n}^Tzz^T1_{m+n}) = (\frac{1}{2}\M_{i,}1_{m+n})^2 \overset{(\ref{R_X})}{=} (P_{i,}1_n)^2 = x_i^2,$$ where the second equality holds because $\M_{i,}$---the $i$-th row of $\M$---is $z_iz^T$. An analogous statement holds for the diagonal entries of $Y$, and hence the condition is necessary.}
	
\jeffr{To show sufficiency, let $z \defeq \xyvec$.} Since $\M$ is psd, we have that $\M_{i,j} \le \sqrt{\M_{i,i}\M_{j,j}}$, which implies $\M_{i,j} \le z_iz_j$ by the assumption of the proposition. \jeffo{Recall from the distribution constraint~(\ref{Eq: Distribution Inequalities}) that $\sum_{i=1}^{m+n}\sum_{j=1}^{m+n} \M_{i,j}=4$. Further, the same constraint along with the definitions of $x$ and $y$ imply that $\sum_{i=1}^{m+n} z_i = 2$, which means that $\sum_{i=1}^{m+n} \sum_{j=1}^{m+n} z_iz_j = 4$. Hence in order to have the equality}
	
	$$4 = \sum_{i=1}^{m+n} \sum_{j=1}^{m+n} \M_{i,j} \le \sum_{i=1}^{m+n} \sum_{j=1}^{m+n} z_iz_j = 4,$$
	we must have $\M_{i,j} = z_iz_j$ for each $i$ and $j$. Consequently $\M$ is rank-1.\end{proof}

We focus now on two nonconvex objectives that as a consequence of the above proposition would return rank-1 solutions:
\begin{prop}\label{nonconvex objectives} All optimal solutions to~\ref{SDP2} with the objective function $\sum_{i=1}^{m+n} \sqrt{\M_{i,i}}$ or $\Tr(\M) - x^Tx - y^Ty$ are rank-1.\end{prop}
\begin{proof} We show that each of these objectives has a specific lower bound which is achieved if and only if the matrix is rank-1.\\
	Observe that since $\M \succeq \xyvec\xyvec^T$, we have $\sqrt{X_{i,i}}\ge x_i$ and $\sqrt{Y_{i,i}} \ge y_i$, and hence $$\sum_{i=1}^{m+n} \sqrt{\M_{i,i}} \ge \sum_{i=1}^{m} x_i + \sum_{i=1}^n y_i= 2.$$
	Further note that $$\Tr(\M) - \xyvec^T\xyvec \ge \xyvec^T\xyvec - \xyvec^T\xyvec = 0.$$
	
	We can see that the lower bounds are achieved if and only if $X_{i,i} = x_i^2$ and $Y_{i,i} = y_i^2$ for all $i$, which by Proposition~\ref{diagonal sufficiency} happens if and only if $\M$ is rank-1.\end{proof}

We refer to our two objective functions in Proposition~\ref{nonconvex objectives} as the ``\emph{square root objective}'' and the ``\emph{diagonal gap objective}'' respectively. While these are both nonconvex, we will attempt to iteratively minimize them by linearizing them through a first order Taylor expansion. For example, at iteration $k$ of the algorithm, $$\sum_{i=1}^{m+n} \sqrt{\M_{i,i}^{(k)}} \simeq \sum_{i=1}^{m+n} \sqrt{\M_{i,i}^{(k-1)}} + \frac{1}{2\sqrt{\M_{i,i}^{(k-1)}}}(\M_{i,i}^{(k)} - \M_{i,i}^{(k-1)}).$$ Note that for the purposes of minimization, this reduces to minimizing $\sum_{i=1}^{m+n} \frac{1}{\sqrt{\M_{i,i}^{{(k-1)}}}}\M_{i,i}^{(k)}$.

In similar fashion, for the second objective function, at iteration $k$ we can make the approximation $$\Tr(\M)-\xyvec^{(k)T}\xyvec^{(k)} \simeq \Tr(\M)-\xyvec^{(k-1)T}\xyvec^{(k-1)T} - 2\xyvec^{(k-1)T}(\xyvec^{(k)}-\xyvec^{(k-1)}).$$ Once again, for the purposes of minimization this reduces to minimizing $\Tr(\M)-2\xyvec^{(k-1)T}\xyvec^{(k)}$. This approach then leads to the following two algorithms.\footnote{An algorithm similar to Algorithm \ref{DG Algorithm} is used in \cite{ibaraki2001rank}.}

\begin{algorithm}[H]
	\caption{Square Root Minimization Algorithm}\label{SR Algorithm}
	\label{SRM}
	\begin{algorithmic}[1]
		\State Let $x^{(0)} = 1_m, y^{(0)} = 1_n, k = 1$.
		\While {!convergence}
		\State Solve~\ref{SDP2} with $\sum_{i=1}^m \frac{1}{\sqrt{x_i^{(k-1)}}}X_{i,i} + \sum_{i=1}^n \frac{1}{\sqrt{y_i^{(k-1)}}}Y_{i,i}$ as the objective, and let \jeffr{$\M^*$} be an optimal solution.
		\State Let $x^{(k)} = diag(X^*), y^{(k)}=diag(Y^*)$.
		\State Let $k = k+1$.
		\EndWhile
	\end{algorithmic}
\end{algorithm}

\begin{algorithm}[H]
	\caption{Diagonal Gap Minimization Algorithm}\label{DG Algorithm}
	\label{DDM}
	\begin{algorithmic}[1]
		\State Let $x^{(0)}=0_m, y^{(0)}=0_n, k=1$.
		\While {!convergence}
		\State Solve~\ref{SDP2} with $\Tr(X) + \Tr(Y) - 2\xyvec^{(k-1)T}\xyvec^{(k)}$ as the objective, and let \jeffr{$\M^*$} be an optimal solution.
		\State \jeffg{Let $x^{(k)} = P^*1_n, y^{(k)}=P^{*T}1_m$.}
		\State Let $k = k+1$.
		\EndWhile
	\end{algorithmic}
\end{algorithm}

\begin{remark}
	Note that the first iteration of both algorithms uses the nuclear norm (i.e. trace) of $\M$ as the objective.
\end{remark}

The square root algorithm has the following property.

\begin{theorem}
	Let $\M^{(1)}, \M^{(2)}, \ldots$ be the sequence of optimal matrices obtained from the square root algorithm. Then the sequence
	\begin{equation}\label{sum sqrt}
	\{\sum_{i=1}^{m+n} \sqrt{\M_{i,i}^{(k)}}\}
	\end{equation}
	is nonincreasing and is lower bounded by two. If it reaches two at some iteration $t$, then the matrix $\M^{(t)}$ is rank-1.
\end{theorem}
\begin{proof}
	Observe that for any $k>1$,
	$$\sum_{i=1}^{m+n} \sqrt{\M_{i,i}^{(k)}} \le \frac{1}{2}\sum_{i=1}^{m+n} (\frac{\M_{i,i}^{(k)}}{\sqrt{\M_{i,i}^{(k-1)}}}+\sqrt{\M_{i,i}^{(k-1)}}) \le \frac{1}{2}\sum_{i=1}^{m+n} (\frac{\M_{i,i}^{(k-1)}}{\sqrt{\M_{i,i}^{(k-1)}}}+\sqrt{\M_{i,i}^{(k-1)}}) = \sum_{i=1}^{m+n} \sqrt{\M_{i,i}^{(k-1)}},$$
	where the first inequality follows from the arithmetic-mean-geometric-mean inequality, and the second follows from that $\M_{i,i}^{(k)}$ is chosen to minimize $\sum_{i=1}^{m+n} \frac{\M_{i,i}^{(k)}}{\sqrt{\M_{i,i}^{(k-1)}}}$ and hence achieves a no larger value than the feasible solution $\M^{(k-1)}$. This shows that the sequence is nonincreasing.
	
	The proof of Proposition~\ref{nonconvex objectives} already shows that the sequence is lower bounded by two, and Proposition~\ref{nonconvex objectives} itself shows that reaching two is sufficient to have the matrix be rank-1.
\end{proof}

The diagonal gap algorithm has the following property.
\begin{theorem} 
	Let $\M^{(1)}, \M^{(2)}, \ldots$ be the sequence of optimal matrices obtained from the diagonal gap algorithm. Then the sequence \begin{equation}\{\Tr(\M^{(k)})- \xyvec^{(k)T}\xyvec^{(k)}\}\end{equation}
	is nonincreasing and is lower bounded by zero. If it reaches zero at some iteration $t$, then the matrix $\M^{(t)}$ is rank-1.
\end{theorem}
\begin{proof} Observe that
	\begin{align*}
	\Tr(\M^{(k)}) - \xyvec^{(k)T}\xyvec^{(k)} &\le \Tr(\M^{(k)})-\xyvec^{(k)T}\xyvec^{(k)} + (\xyvec^{(k)}- \xyvec^{(k-1)})^T(\xyvec^{(k)}- \xyvec^{(k-1)})\\
	&=\Tr(\M^{(k)})-2\xyvec^{(k)T}\xyvec^{(k-1)}+\xyvec^{(k-1)T}\xyvec^{(k-1)}\\
	&\le \Tr(\M^{(k-1)})-2\xyvec^{(k-1)T}\xyvec^{(k-1)}+\xyvec^{(k-1)T}\xyvec^{(k-1)}\\
	&=\Tr(\M^{(k-1)})-\xyvec^{(k-1)T}\xyvec^{(k-1)},
	\end{align*}
	where the second inequality follows from that \jeffr{$\M^{(k)}$} is chosen to minimize $$\Tr(\M^{(k-1)})-2\xyvec^{(k-1)T}\xyvec^{(k-1)}$$ and hence achieves a no larger value than the feasible solution \jeffr{$\M^{(k-1)}$}. This shows that the sequence is nonincreasing.
	
	The proof of Proposition~\ref{nonconvex objectives} already shows that the sequence is lower bounded by zero, and Proposition~\ref{nonconvex objectives} itself shows that reaching zero is sufficient to have the matrix be rank-1.\end{proof}
	
	\jeffg{We also invite the reader to also see Theorem \ref{Thm: Trace minus 2xxt} in the next section which relates the objective value of the diagonal gap minimization algorithm and the quality of approximate Nash equilibria that the algorithm produces.}

\subsection{Numerical Experiments}
We tested Algorithms~\ref{SR Algorithm} and~\ref{DG Algorithm} on games coming from 100 randomly generated payoff matrices with entries bounded in $[0,1]$ of varying sizes. Below is a table of statistics for $20 \times 20$ matrices; the data for the rest of the sizes can be found in Appendix~\ref{Appendex Statistics}.\footnote{The code that produced these results is publicly available at \url{aaa.princeton.edu/software}. The function \texttt{nash.m} computes an approximate Nash equilibrium using one of our two algorithms as specified by the user.} We can see that our algorithms return approximate Nash equilibria with fairly low $\epsilon$ (recall the definition from Section~\ref{Nash as QP}). We ran 20 iterations of each algorithm on each game. {\aaa Using the SDP solver of MOSEK~\cite{mosek}, each iteration takes on average under 4 seconds to solve on a standard personal machine with a 3.4 GHz processor and 16 GB of memory.}

\begin{table}[H]
	\caption{Statistics on $\epsilon$ for $20\times 20$ games after 20 iterations.\label{5x5}}
	{\begin{tabular}{|c|c|c|c|c|c|}\hline
			Algorithm & Max & Mean & Median & StDev\\\hline
			Square Root&	0.0198&	0.0046&	0.0039&	0.0034	\\\hline
			Diagonal Gap&	0.0159&	0.0032&	0.0024&	0.0032	\\\hline
	\end{tabular}}
	\centering
\end{table}

The histograms below show the effect of increasing the number of iterations on lowering $\epsilon$ on $20 \times 20$ games. For both algorithms, there was a clear improvement of the $\epsilon$ by increasing the number of iterations.

\begin{figure}[H]
	\includegraphics[height=.3\textheight,keepaspectratio]{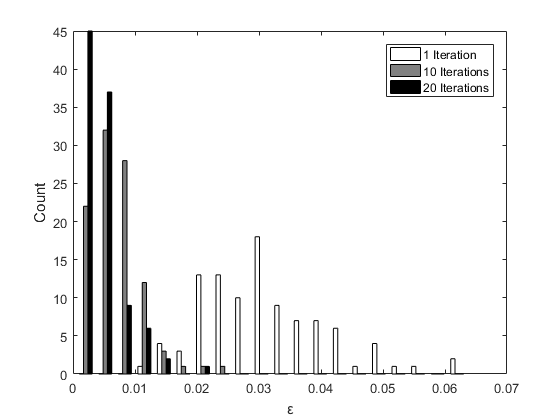}
	\includegraphics[height=.3\textheight,keepaspectratio]{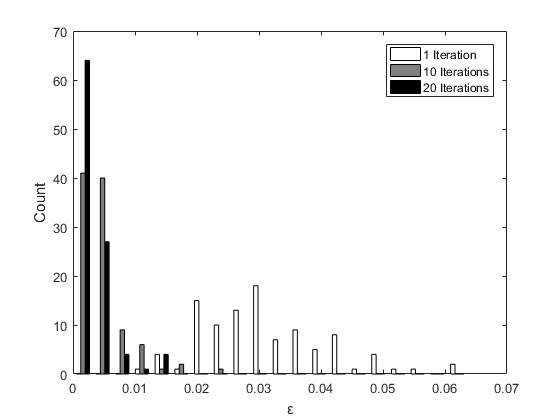}
	\caption{Distribution of $\epsilon$ over numbers of iterations for the square root algorithm (left) and the diagonal gap algorithm (right).\label{Epsilon Iterations}}
\end{figure}

\section{Bounds on $\epsilon$ for General Games}\label{Bounds}

Since the problem of computing a Nash equilibrium to an arbitrary bimatrix game is PPAD-complete, it is unlikely that one can find rank-1 solutions to this SDP in polynomial time. In Section~\ref{Algorithms}, we design\jeffg{ed} objective functions (such as variations of the nuclear norm) that empirically do very well in finding low-rank solutions to~\ref{SDP2}. Nevertheless, it is of interest to know if the solution returned by~\ref{SDP2} is not rank-1, whether one can recover an $\epsilon$-Nash equilibrium from it and have a guarantee on $\epsilon$. Our goal in this section is to study this question.

\jeffg{\emph{Notational Remark:}} Recall our notation for the matri\jeffr{x}

$$\mathcal{M} \mathrel{\mathop:}= \left[\begin{matrix} X & P\\ Z & Y\end{matrix}\right].$$
Throughout this section, any matrices $X, Z, P=Z^T$ and $Y$ are assumed to be taken from a feasible solution to~\ref{SDP2}. \jeffo{Furthermore, $x$ and $y$ will be $P1_n$ and $P^T1_m$ respectively.}

The ultimate results of this section are \jeffg{the theorems in Sections \ref{SSec: Bounds on e} and \ref{SSec: Rank-2 case}}. To work towards them, we need a number of preliminary lemmas which we present in Section~\ref{Useful Lemmas}.

\subsection{Lemmas Towards Bounds on $\epsilon$}\label{Useful Lemmas}

We first observe the following connection between the approximate payoffs $\Tr(AZ)$ and $\Tr(BZ)$, and $\epsilon(x,y)$, as defined in Section~\ref{Nash as QP}.
\begin{lem}\label{Bound for e} \jeffo{Consider any feasible solution to~\ref{SDP2}.} Then
	$$\epsilon(x,y) \le \max\{\Tr(AZ)-x^TAy, \Tr(BZ) - x^TBy\}.$$\end{lem}
\begin{proof} \jeffg{Recall from the argument at the beginning of Section \ref{SSec: Simplification} that constraints (\ref{CE_A}) and (\ref{CE_B}) imply $\Tr(AZ) \ge e_i^TAy$ and $\Tr(BZ) \ge x^TBe_i$ for all $i$.} Hence, we have $\epsilon_A \le \Tr(AZ)-x^TAy$ and $\epsilon_B \le \Tr(BZ) - x^TBy$.\end{proof}

We thus are interested in the difference of the two matrices $P=Z^T$ and $xy^T$. These two matrices can be interpreted as two different probability distributions over the strategy outcomes. The matrix $P$ is the probability distribution from the SDP which generates the approximate payoffs $\Tr(AZ)$ and $\Tr(BZ)$, while $xy^T$ is the product distribution that would have resulted if the matrix had been rank-1. We will see that the difference of these distributions is key in studying the $\epsilon$ which results from \ref{SDP2}. Hence, we first take steps to represent this difference.

\begin{lem}\label{Partition Lemma} \jeffr{Consider any feasible matrix $\M$ to~\ref{SDP2} with an eigendecomposition}
	\begin{equation}
	\begin{aligned}
	\M = \sum_{i=1}^k \lambda_i v_iv_i^T =\mathrel{\mathop:} \sum_{i = 1}^k \lambda_i \bmat a_i \\ b_i\emat \bmat a_i \\ b_i\emat^T,
	\end{aligned}
	\end{equation}
	so that the eigenvectors $v_i\in \mathbb{R}^{m+n}$ are partitioned into vectors $a_i \in \mathbb{R}^m$ and $b_i\in \mathbb{R}^n$. Then for all $i,\sum_{j=1}^m (a_i)_j = \sum_{j=1}^n (b_i)_j$.\end{lem}
\begin{proof}
	We know from \jeffo{(\ref{SDP2 Distribution}), (\ref{SDP2 Row x}), and (\ref{SDP2 Row y})} that
	\begin{align}
	\sum_{i=1}^k \lambda_i 1_m^T a_ia_i^T 1_m \overset{(\ref{SDP2 Distribution}),(\ref{SDP2 Row x})}{=} 1, \label{pl-1}\\
	\sum_{i=1}^k \lambda_i 1_m^T a_ib_i^T 1_n \overset{(\ref{SDP2 Distribution})}{=} 1,\label{pl-2}\\
	\sum_{i=1}^k \lambda_i 1_n^T b_ia_i^T 1_m \overset{(\ref{SDP2 Distribution})}{=} 1,\label{pl-3}\\
	\sum_{i=1}^k \lambda_i 1_n^T b_ib_i^T 1_n \overset{(\ref{SDP2 Distribution}),(\ref{SDP2 Row y})}{=} 1.\label{pl-4}
	\end{align}
	
	Then by subtracting terms we have
	\begin{align}
	(\ref{pl-1})-(\ref{pl-2})=\sum_{i=1}^k \lambda_i 1_m^Ta_i (a_i^T1_m - b_i^T1_n) = 0, \label{pl-5}\\
	(\ref{pl-3})-(\ref{pl-4})=\sum_{i=1}^k \lambda_i 1_n^Tb_i (a_i^T1_m - b_i^T1_n) = 0. \label{pl-6}
	\end{align}
	
	By subtracting again these imply
	\begin{equation}
	(\ref{pl-5})-(\ref{pl-6})=\sum_{i=1}^k \lambda_i (1_m^T a_i - 1_n^Tb_i)^2 = 0.\label{pl-7}
	\end{equation}
	As all $\lambda_{i}$ are nonnegative due to positive semidefiniteness of $\M$, the only way for this equality to hold is to have $1_m^Ta_i = 1_n^Tb_i, \forall i$. This is equivalent to the statement of the claim.
\end{proof}

From Lemma~\ref{Partition Lemma}, we can let $s_i \defeq \sum_{j=1}^m (a_i)_j = \sum_{j=1}^n (b_i)_j$, and furthermore we assume without loss of generality that each $s_i$ is nonnegative. \jeffo{Note that from the definition of $x$} we have
\begin{equation} x_i = \sum_{j = 1}^m P_{ij} = \sum_{l = 1}^k \sum_{j = 1}^m \lambda_l (a_l)_i (b_l)_j = \sum_{j = 1}^k \lambda_j s_j (a_l)_i.
\end{equation}
Hence,
\begin{equation}\label{Row constraint x}x = \sum_{i =1}^k \lambda_i s_i a_i.\end{equation}
Similarly,
\begin{equation}\label{Row constraint y}
y = \sum_{i=1}^k \lambda_i s_i b_i.
\end{equation}
Finally note from the distribution constraint (\ref{Eq: Distribution Inequalities}) that this implies \begin{equation}\label{dist const}
\sum_{i=1}^k \lambda_is_i^2 = 1.\end{equation}

\begin{lem}\label{Representation Lemma}
	Let $$\M = \sum_{i=1}^k \lambda_i \bmat a_i\\b_i \emat \bmat a_i\\b_i \emat^T,$$ be a feasible solution to~\ref{SDP2}, such that the eigenvectors of $\M$ are partitioned into $a_i$ and $b_i$ with $\sum_{j=1}^m (a_i)_j = \sum_{j=1}^n (b_i)_j=s_i, \forall i$. Then $$P-xy^T= \sum_{i = 1}^k \sum_{j>i}^k \lambda_i\lambda_j (s_ja_i - s_ia_j)(s_jb_i - s_ib_j)^T.$$
\end{lem}
\begin{proof}
	Using equations (\ref{Row constraint x}) and (\ref{Row constraint y}) we can write
	\begin{equation*}
	\begin{aligned}
	P - xy^T &= \sum_{i=1}^k \lambda_i a_ib_i^T - (\sum_{i =1}^k \lambda_i s_i a_i)(\sum_{j=1}^k \lambda_j s_j b_j)^T\\
	&= \sum_{i=1}^k \lambda_ia_i(b_i - s_i\sum_{j=1}^k \lambda_js_jb_j)^T\\
	&\overset{(\ref{dist const})}{=}\sum_{i=1}^k \lambda_ia_i(\sum_{j=1}^k \lambda_js_j^2b_i - s_i\sum_{j=1}^k \lambda_js_jb_j)^T\\
	&=\sum_{i=1}^k\sum_{j=1}^k \lambda_i\lambda_j a_is_j(s_jb_i-s_ib_j)^T\\
	&=\sum_{i = 1}^k \sum_{j > i}^k \lambda_i\lambda_j (s_ja_i - s_ia_j)(s_jb_i - s_ib_j)^T,
	\end{aligned}
	\end{equation*}
	\jeff{where the last line follows from observing that terms where $i$ and $j$ are switched can be combined.}
\end{proof}

We can relate $\epsilon$ and $P-xy^T$ with the following lemma.
\begin{lem}\label{L1 norm/2 bound}
	\jeffg{Let the matrix $P$ and the vectors $x \defeq P1_n$ and $y \defeq P^T1_m$ come from any feasible solution to $\ref{SDP2}$.} Then $$\epsilon \le \frac{\|P-xy^T\|_1}{2},$$ where $\|\cdot\|_1$ here denotes the entrywise L-1 norm, i.e., the sum of the absolute values of the entries of the matrix.
\end{lem}
\begin{proof}
	Let $D \mathrel{\mathop:}= P - xy^T$. From Lemma~\ref{Bound for e}, $$\epsilon_A \le \Tr(AZ) - x^TAy = \Tr(A(Z-yx^T)).$$ If we then hold $D$ fixed and restrict that $A$ has entries bounded in [0,1], the quantity $\Tr(AD^T)$ is maximized when $$A_{i,j} = \begin{cases} 1 & D_{i,j} \ge 0 \\ 0 & D_{i,j} < 0\end{cases}.$$ The resulting quantity $\Tr(AD^T)$ will then be the sum of all nonnegative elements of $D$. Since the sum of all elements in $D$ is zero, this quantity will be equal to $\frac{1}{2}\|D\|_1$.\\
	The proof for $\epsilon_B$ is identical, and the result follows from that $\epsilon$ is the maximum of $\epsilon_A$ and $\epsilon_B$.
\end{proof}

\subsection{Bounds on $\epsilon$}\label{SSec: Bounds on e}

We provide a number of bounds on $\epsilon(x,y)$ \jeffr{for $x \defeq P1_n$ and $y \defeq P^T1_m$ coming from any feasible solution to \ref{SDP2}}. \jeffg{Our first two theorems roughly state that} solutions which are ``close'' to rank-1 provide small $\epsilon$.

\begin{theorem}\label{nksl}
	\jeffo{Consider any feasible solution $\M$ to~\ref{SDP2}. Suppose $\M$ is rank-$k$ and its eigenvalues are $\lambda_1 \ge \lambda_2 \ge ... \ge \lambda_k > 0$. Then $x$ and $y$ constitute an $\epsilon$-NE to the game $(A,B)$ with $\epsilon \le \frac{m+n}{2}\sum_{i=2}^k \lambda_i.$}
\end{theorem}
\begin{proof}
	By the Perron Frobenius theorem (see e.g.~\cite[Chapter 8.3]{meyer2000matrix}), the eigenvector corresponding to $\lambda_1$ can be assumed to be nonnegative, and hence \begin{equation}\label{s1 value}
	s_1 = \|a_1\|_1 = \|b_1\|_1.
	\end{equation}
	We further note that for all $i$, since $\bmat a_i\\b_i\emat$ is a vector of length $m+n$ with 2-norm equal to 1, we must have
	\begin{equation}\label{L1 norm sum}\left\|\bmat a_i\\b_i\emat\right\|_1 \le \sqrt{m+n}.\end{equation}
	Since $s_i$ is the sum of the elements of $a_i$ and $b_i$, we know that
	\begin{equation}\label{si bound} s_i \le \min\{\|a_i\|_1, \|b_i\|_1\} \le \frac{\sqrt{m+n}}{2}.\end{equation}
	This then gives us
	\begin{equation}\label{L1 norm prod}s_i^2 \le \|a_i\|_1\|b_i\|_1 \le \frac{m+n}{4},\end{equation}
	\jeff{with the first inequality following from (\ref{si bound}) and the second from (\ref{L1 norm sum}).}
	Finally note that a consequence of the nonnegativity of $\|\cdot\|_1$ and (\ref{L1 norm sum}) is that for all $i, j$,
	\begin{equation}\label{l1 cross product}
	\|a_i\|_1\|b_j\|_1 + \|b_i\|_1\|a_j\|_1 \le \jeff{(\|a_i\|_1+\|b_i\|_1)(\|a_j\|_1+\|b_j\|_1)=} \left\|\bmat a_i\\b_i\emat\right\|_1\left\|\bmat a_j\\b_j\emat\right\|_1 \overset{(\ref{L1 norm sum})}{\le} m+n.
	\end{equation}
	Now we let $D \mathrel{\mathop:}= P - xy^T$ and upper bound $\frac{1}{2}\|D\|_1$ using Lemma~\ref{Representation Lemma}.
	\begin{align}
	\frac{1}{2}\|D\|_1 &= \frac{1}{2}\|\sum_{i = 1}^k \sum_{j>i}^k \lambda_i\lambda_j (s_ja_i - s_ia_j)(s_jb_i - s_ib_j)^T\|_1\nonumber \\
	&\le \frac{1}{2}\sum_{i = 1}^k \sum_{j>i}^k \|\lambda_i\lambda_j (s_ja_i - s_ia_j)(s_jb_i - s_ib_j)^T\|_1\nonumber  \\
	& \le \frac{1}{2}\sum_{i = 1}^k \sum_{j > i}^k \lambda_i\lambda_j \|s_ja_i - s_ia_j\|_1\|s_jb_i - s_ib_j\|_1\nonumber \\
	& \le \frac{1}{2}\sum_{i = 1}^k \sum_{j > i}^k \lambda_i\lambda_j (s_j\|a_i\|_1 + s_i\|a_j\|_1)(s_j\|b_i\|_1 + s_i\|b_j\|_1) \label{Eq: Hell} \\
	& \overset{(\ref{s1 value}),(\ref{L1 norm prod})}{\le} \frac{1}{2}\sum_{j=2}^k\lambda_1s_1^2\lambda_j(s_j+\|a_j\|_1)(s_j+\|b_j\|_1)\nonumber  \\
	&+\frac{1}{2}\sum_{i = 2}^k \sum_{j > i}^k \lambda_i\lambda_j (s_j^2\frac{m+n}{4} + s_i^2\frac{m+n}{4} + s_is_j\|a_i\|_1\|b_j\|_1+ s_is_j\|a_j\|_1\|b_i\|_1)\nonumber \\
	&\overset{(\ref{L1 norm sum}),(\ref{l1 cross product}),(\ref{si bound})}{\le} \frac{m+n}{2}\lambda_1s_1^2 \sum_{i=2}^k \lambda_i\nonumber \\
	&+ \frac{1}{2}\sum_{i=2}^k \sum_{j > i}^k \lambda_i\lambda_j\frac{m+n}{4}(s_i^2 + s_j^2) + \lambda_i\lambda_js_is_j (m+n)\nonumber \\
	&\overset{\text{AMGM\footnotemark}}{\le} \frac{m+n}{2}\lambda_1s_1^2 \sum_{i=2}^k \lambda_i + \frac{m+n}{2}\sum_{i=2}^k \sum_{j > i}^k \lambda_i\lambda_j (\frac{s_i^2 + s_j^2}{4} + \frac{s_i^2+s_j^2}{2})\nonumber \\
	&= \frac{m+n}{2}\lambda_1s_1^2 \sum_{i=2}^k \lambda_i + \frac{3(m+n)}{8}\sum_{i=2}^k \sum_{j > i}^k \lambda_i\lambda_j (s_i^2+s_j^2)\nonumber \\
	&= \frac{m+n}{2}\lambda_1s_1^2 \sum_{i=2}^k \lambda_i + \frac{3(m+n)}{8}(\sum_{i=2}^k\lambda_is_i^2 \sum_{j > i}^k \lambda_j + \sum_{i=2}^k \lambda_i \sum_{j > i}^k \lambda_js_j^2)\nonumber \\
	&= \frac{m+n}{2}\lambda_1s_1^2 \sum_{i=2}^k \lambda_i + \frac{3(m+n)}{8}(\sum_{j = 2}^k \lambda_j \sum_{2\le i<j}^k\lambda_is_i^2 + \sum_{i=2}^k \lambda_i \sum_{j > i}^k \lambda_js_j^2)\nonumber \\
	&\le \frac{m+n}{2}\lambda_1s_1^2 \sum_{i=2}^k \lambda_i + \frac{3(m+n)}{8}(\sum_{j =2}^k \lambda_js_j^2)\sum_{i=2}^k \lambda_i\nonumber \\
	&\overset{(\ref{dist const})}{=}
	\frac{m+n}{2}\lambda_1s_1^2 \sum_{i=2}^k \lambda_i + \frac{3(m+n)}{8}(1-\lambda_1s_1^2)\sum_{i=2}^k \lambda_i\nonumber \\
	&= \frac{m+n}{8}(3+\lambda_1s_1^2)\sum_{i=2}^k \lambda_i\nonumber \\
	&\overset{(\ref{dist const})}{\le} \frac{m+n}{2} \sum_{i=2}^k \lambda_i.\nonumber 
	\end{align}
\end{proof}
\footnotetext{AMGM is used to denote the arithmetic-mean-geometric-mean inequality.}

\jeffg{
	The following theorem quantifies how making the objective of the diagonal gap algorithm from Section \ref{Algorithms} small makes $\epsilon$ small. The proof is similar to the proof of Theorem~\ref{nksl}.
	\begin{theorem}\label{Thm: Trace minus 2xxt}
		Let $\M$ be a feasible solution to~\ref{SDP2}. Then, $x$ and $y$ constitute an $\epsilon$-NE to the game $(A,B)$ with $\epsilon \le \frac{3(m+n)}{8}(\Tr(\M)-x^Tx-y^Ty).$
	\end{theorem}
	\begin{proof}
		Let $\M$ be rank-$k$ with eigenvalues $\lambda_1 \ge \lambda_2 \ge \ldots \ge \lambda_k > 0$ and eigenvectors $v_1, \ldots, v_k$ partitioned as in Lemma \ref{Partition Lemma} so that $v_i = \bmat a_i\\ b_i\emat$ with $\sum_{j=1}^m (a_i)_j = \sum_{j=1}^n (b_i)_j$ for $i = 1, \ldots, k$. Let $s_i \defeq \sum_{j=1}^m (a_i)_j$. Then we have $\Tr(\M) = \sum_{i=1}^k \lambda_i$, and
		\begin{equation}
		\label{Eq: xynorm}
		x^Tx + y^Ty \overset{(\ref{Row constraint x}),(\ref{Row constraint y})}{=} (\sum_{i=1}^k \lambda_i s_i v_i)^T(\sum_{i=1}^k \lambda_i s_i v_i) = \sum_{i=1}^k \lambda_i^2s_i^2.\end{equation}
		We now get the following chain of inequalities (the first one follows from Lemma~\ref{L1 norm/2 bound} and inequality~(\ref{Eq: Hell})):
		\begin{align*}
		\epsilon &\le \frac{1}{2}\sum_{i = 1}^k \sum_{j > i}^k \lambda_i\lambda_j (s_j\|a_i\|_1 + s_i\|a_j\|_1)(s_j\|b_i\|_1 + s_i\|b_j\|_1)\\
		& \overset{(\ref{s1 value}),(\ref{L1 norm prod})}{\le} \frac{1}{2}\sum_{i = 1}^k \sum_{j > i}^k \lambda_i\lambda_j (s_j^2\frac{m+n}{4} + s_i^2\frac{m+n}{4} + s_is_j\|a_i\|_1\|b_j\|_1+ s_is_j\|a_j\|_1\|b_i\|_1)\\
		&\overset{(\ref{l1 cross product})}{\le} \frac{1}{2}\sum_{i=1}^k \sum_{j > i}^k \lambda_i\lambda_j\frac{m+n}{4}(s_i^2 + s_j^2) + \lambda_i\lambda_js_is_j (m+n)\\
		&\overset{AMGM}{\le} \frac{m+n}{2}\sum_{i=1}^k \sum_{j > i}^k \lambda_i\lambda_j (\frac{s_i^2 + s_j^2}{4} + \frac{s_i^2+s_j^2}{2})\\
		&= \frac{3(m+n)}{8}\sum_{i=1}^k \sum_{j > i}^k \lambda_i\lambda_j (s_i^2+s_j^2)\\
		&= \frac{3(m+n)}{8}(\sum_{i=1}^k\lambda_is_i^2 \sum_{j > i}^k \lambda_j + \sum_{i=1}^k \lambda_i \sum_{j > i}^k \lambda_js_j^2)\\
		&= \frac{3(m+n)}{8}(\sum_{j = 1}^k \lambda_j \sum_{1\le i<j}^k\lambda_is_i^2 + \sum_{i=1}^k \lambda_i \sum_{j > i}^k \lambda_js_j^2)\\
		&= \frac{3(m+n)}{8}(\sum_{i = 1}^k \lambda_i \sum_{j \ne i}\lambda_js_j^2)\\
		&\overset{(\ref{dist const})}{=} \frac{3(m+n)}{8}(\sum_{i = 1}^k \lambda_i (1-\lambda_is_i^2))\\
		&= \frac{3(m+n)}{8}(\sum_{i=1}^k \lambda_i - \sum_{i=1}^k \lambda_i^2s_i^2) \overset{(\ref{Eq: xynorm})}{=} \frac{3(m+n)}{8}(\Tr(\M)-x^Tx-y^Ty).
		\end{align*}
	\end{proof}
}

\jeffo{We now give a bound on $\epsilon$ which is dependent on the nonnegative rank of the matrix returned by \ref{SDP2}. Our analysis will also be useful for the next subsection. To begin, we first recall the definition of the nonnegative rank.

\begin{defn}\label{Defn: Nonnegative Rank}
	The \emph{nonnegative rank} of a (nonnegative) $m \times n$ matrix $M$ is the smallest $k$ for which there exist a nonnegative $m \times k$ matrix $U$ and a nonnegative $n \times k$ matrix $V$ such that $M = UV^T$. Such a decomposition is called a \emph{nonnegative matrix factorization} of $M$.
\end{defn}

\begin{theorem}\label{Thm: 1-1/k}
	Consider the matrix $P$ from any feasible solution to \ref{SDP2}. Suppose its nonnegative rank is $k$. Then $x \defeq P1_n$ and $y \defeq P^T1_m$ constitute an $\epsilon$-NE to the game $(A,B)$ with $\epsilon \le 1-\frac{1}{k}$.
\end{theorem}
\begin{proof}


Since $P$ has nonnegative rank $k$ and its entries sum up to 1, we can write $P = \sum_{i=1}^k \sigma_ia_ib_i^T$, where $a_i \in \triangle_m, b_i \in \triangle_n$, and $\sum_{i=1}^k \sigma_i = 1$. From Lemma~\ref{L1 norm/2 bound} and inequality~(\ref{Eq: Hell}) (keeping in mind that $s_i = 1,\ \forall\ i$) we have
\begin{align*}
\epsilon &\le \frac{1}{2} \sum_{i=1}^k \sum_{j > i}^k \sigma_i\sigma_j (\|a_i\|_1 + \|a_j\|_1)(\|b_i\|_1 + \|b_j\|_1)\\
& \le 2 \sum_{i=1}^k \sum_{j > i}^k \sigma_i\sigma_j\\
&= 2(\frac{1}{2}(\sum_{i=1}^k \sigma_i \sum_{j=1}^k \sigma_j - \sum_{i=1}^k \sigma_i^2))\\
&= 1 - \sum_{i=1}^k \sigma_i^2\\
&\le 1 - \frac{1}{k},
\end{align*}
where the last line follows from  the fact that $\|v\|_2^2 \ge \frac{1}{k}$ for any vector $v \in \triangle_k$.
\end{proof}}

\subsection{Bounds on $\epsilon$ in the Rank-2 Case}\label{SSec: Rank-2 case}

\jeffg{We now provide a number of bounds on $\epsilon(x,y)$ with $x \defeq P1_n$ and $y \defeq P^T1_m$ which hold for rank-2 feasible solutions $\M$ to \ref{SDP2} (note that $P$ will have rank at most 2 in this case).} This is motivated by our ability to show stronger (constant) bounds in this case, and the fact that we often recover rank-2 (or rank-1) solutions with our algorithms in Section~\ref{Algorithms}. \jeffo{Furthermore, our analysis will use the special property that a rank-2 nonnegative matrix will have nonnegative rank also equal to two, and that a nonnegative factorization of it can be computed in polynomial time (see, e.g., Section 4 of \cite{cohen1993nonnegative}). We begin with the following observation, which follows from Theorem \ref{Thm: 1-1/k} when $k = 2$.}

\begin{cor}\label{1/2e} \jeffg{If the matrix $P$ from a feasible solution to~\ref{SDP2} is rank-2, then $x$ and $y$ constitute a $\frac{1}{2}-$NE.}\end{cor}

We now show how this pair of strategies can be refined.

\begin{theorem}\label{5/11e}
\jeffg{If the matrix $P$ from a feasible solution to~\ref{SDP2} is rank-2, then either $x$ and $y$ constitute a $\frac{5}{11}$-NE, or a $\frac{5}{11}$-NE can be recovered from $P$ in polynomial time.}
\end{theorem}
\begin{proof}
	We consider 3 cases, depending on whether $\epsilon_A(x,y)$ and $\epsilon_B(x,y)$ are greater than or less than .4. If $\epsilon_A \le .4, \epsilon_B \le .4$, then $(x, y)$ is already a .4-Nash equilibrium. Now consider the case when $\epsilon_A \ge .4, \epsilon_B \ge .4$. Since $\epsilon_A \le \Tr(A(P-xy^T)^T)$ and $\epsilon_B \le \Tr(B(P-xy^T)^T)$ \jeffg{as seen in the proof of Lemma \ref{Bound for e}, we have, reusing the notation in the proof of Theorem \ref{Thm: 1-1/k},}
	\jeffo{\begin{equation*}
	\sigma_1\sigma_2(a_1-a_2)^T A (b_1-b_2) \ge .4,
	\sigma_1\sigma_2(a_1-a_2)^T B (b_1-b_2) \ge .4.
	\end{equation*}
	Since $A, a_1, a_2, b_1,$ and $b_2$ are all nonnegative and $\sigma_1\sigma_2 \le \frac{1}{4}$,
	$$a_1^TAb_1 + a_2^TAb_2 \ge (a_1-a_2)^T A (b_1-b_2) \ge 1.6,$$ and the same inequalities hold for for player B. In particular, since $A$ and $B$ have entries bounded in [0,1] and $a_1,a_2,b_1,$ and $b_2$ are simplex vectors, all the quantities $a_1^TAb_1, a_2^TAb_2, a_1^TBb_1,\ \text{and}\ a_2^TBb_2$ are at most 1, and consequently at least .6. Hence $(a_1,a_2)$ and $(a_2,b_2)$ are both .4-Nash equilibria.}
	
	Now suppose that $(x,y)$ is a .4-NE for one player (without loss of generality player A) but not for the other (without loss of generality player B). Then $\epsilon_A \le .4$, and $\epsilon_B \ge .4$. Let $y^*$ be a best response for player B to $x$, and let $p = \frac{1}{1+\epsilon_B - \epsilon_A}$. Consider the strategy profile $(\tilde{x},\tilde{y}) \defeq (x, py + (1-p)y^*)$. This can be interpreted as the outcome $(x,y)$ occurring with probability $p$, and the outcome $(x,y^*)$ happening with probability $1-p$. In the first case, player A will have $\epsilon_A(x,y) = \epsilon_A$ and player B will have $\epsilon_B(x,y) = \epsilon_B$. In the second outcome, player A will have $\epsilon_A(x,y^*)$ at most 1, while player~B will have $\epsilon_B(x,y^*) = 0$. Then under this strategy profile, both players have the same upper bound for $\epsilon$, which equals $\epsilon_B p = \frac{\epsilon_B}{1 + \epsilon_B - \epsilon_A}$. To find the worst case for this value, let $\epsilon_B = .5$ (note from Theorem~\ref{1/2e} that $\epsilon_B \le \frac{1}{2}$) and $\epsilon_A = .4$, and this will return $\epsilon = \frac{5}{11}$.
	
\end{proof}

We now show a stronger result in the case of symmetric games.

\begin{defn} A \emph{symmetric game} is a game in which the payoff matrices $A$ and $B$ satisfy $B=A^T$.\end{defn}\label{Defn: Symmetric Game}
\begin{defn} A Nash equilibrium strategy $(x,y)$ is said to be \emph{symmetric} if $x=y$.\end{defn}
\begin{theorem}[see Theorem 2 in~\cite{nash1951}]	Every symmetric bimatrix game has a symmetric Nash equilibrium.
\end{theorem}

For the proof of Theorem~\ref{1/3se} below we modify \ref{SDP2} so that we are seeking a symmetric solution. We also need a more specialized notion of the nonnegative rank.

\begin{defn} A matrix $M$ is \emph{completely positive} (CP) if it admits a decomposition $M=UU^T$ for some nonnegative matrix $U$.\end{defn}\label{Defn: Completely Positive}
\begin{defn} The \emph{CP-rank} of an $n \times n$ CP matrix $M$ is the smallest $k$ for which there exists a nonnegative $n \times k$ matrix $U$ such that $M = UU^T$.\end{defn}
\begin{theorem}[see e.g.~\cite{kalofolias2012computing} or Theorem 2.1 in~\cite{berman2003completely}]\label{CP rank 2}
	A rank-2, nonnegative, and positive semidefinite matrix is CP and has CP-rank 2.
\end{theorem}

It is also known (see e.g., Section 4 in~\cite{kalofolias2012computing}) that the CP factorization of a rank-2 CP matrix can be found to arbitrary accuracy in polynomial time.

\begin{theorem}\label{1/3se} \jeffo{Suppose the constraint $P \succeq 0$ is added to~\ref{SDP2}. Then if \jeffg{in a feasible solution to this new SDP the matrix $P$} is rank-2, either $x$ and $y$ constitute a symmetric $\frac{1}{3}$-NE, or a symmetric $\frac{1}{3}$-NE can be recovered from $P$ in polynomial time.}\end{theorem} 
\begin{proof}
	If $(x,y)$ is already a symmetric $\frac{1}{3}$-NE, then the claim is established. Now suppose that $(x,y)$ does not constitute a $\frac{1}{3}$-Nash equilibrium. \jeffo{Similarly as in the proof of Theorem \ref{Thm: 1-1/k}, we can decompose $P$ into $\sum_{i=1}^2 \sigma_i a_ia_i^T$, where $\sum_{i=1}^2 \sigma_i = 1$ and each $a_i$ is a vector on the unit simplex. Then we have
	\begin{equation*}
	\sigma_1\sigma_2(a_1-a_2)^T A (a_1-a_2) \ge \frac{1}{3}.
	\end{equation*}
	Since $A, a_1,$ and $a_2$ are all nonnegative, and $\sigma_1\sigma_2 \le \frac{1}{4}$, we get
	\begin{equation*}a_1^TAa_1+ a_2^TAa_2 \ge (a_1-a_2)^T A (a_1-a_2) \ge \frac{4}{3}.\end{equation*}}
	In particular, at least one of $a_1^TAa_1$ and $a_2^TAa_2$ is at least $\frac{2}{3}$. Since the maximum possible payoff is 1, at least one of $(a_1,a_1)$ and $(a_2,a_2)$ is a (symmetric) $\frac{1}{3}$-Nash equilibrium.
\end{proof}

\jeffr{\begin{remark}
	For symmetric games, instead of the construction stated in Theorem \ref{1/3se}, one can simply optimize over a smaller $m \times m$ matrix (note $m=n$). This is the relaxed version of exchangeable equilibria~\cite{stein1943exchangeable}, with the completely positive constraint relaxed to a psd constraint.
\end{remark}
}

\jeffg{\begin{remark}\label{Rem: Rank of M}

	The statements of Corollary \ref{1/2e}, and Theorem \ref{5/11e}, and Theorem \ref{1/3se} hold for any rank-2 correlated equilibrium. Indeed, given any rank-2 (equivalently, nonnegative-rank-2) correlated equilibrium $P$, one can complete it to a (rank-2) feasible solution to \ref{SDP2} as follows. Let $P = \sum_{i=1}^2 \sigma_ia_ib_i^T$, where $a_i \in \triangle_m, b_i \in \triangle_n$, and $\sigma_1 + \sigma_2 = 1$. It is easy to check that $$\M~\defeq~\sum_{i=1}^2 \sigma_i \bmat a_i \\ b_i \emat \bmat a_i \\ b_i \emat^T$$ is feasible to \ref{SDP2}.	
\end{remark}}

\section{Bounding Payoffs and Strategy Exclusion in Symmetric Games}\label{Bounding Payoffs and Strategy Exclusion}
In addition to finding $\epsilon$-additive Nash equilibria, our SDP approach can be used to answer certain questions of economic interest about Nash equilibria without actually computing them. For instance, economists often would like to know the maximum welfare (sum of the two players' payoffs) achievable under any Nash equilibrium, or whether there exists a Nash equilibrium in which a given subset of strategies (corresponding, e.g., to undesirable behavior) is not played. Both these questions are NP-hard for bimatrix games \cite{gilboa1989nash}\jeffo{, even when the game is symmetric and only symmetric equilibria are considered \cite{conitzer2008new}}. \jeffg{In this section, we consider these two problems in the symmetric setting and compare the performance of our SDP approach to an LP approach which searches over symmetric correlated equilibria. For general equilibria, it turns out that for these two specific questions, our SDP approach is equivalent to an LP that searches over correlated equilibria.}



\subsection{Bounding Payoffs}\label{Bounding Payoffs}
When designing policies that are subject to game theoretic behavior by agents, economists would often like to find one with a good socially optimal outcome, which usually corresponds to an equilibrium giving the maximum welfare. Hence, given a game, it is of interest to know the highest achievable welfare under any Nash equilibrium. \jeffg{For symmetric games, symmetric equilibria are of particular interest as they reflect the notion that identical agents should behave similarly given identical options.}

\jeffg{Note that the maximum welfare of a symmetric game under any symmetric Nash equilibrium is equal to the optimal value of the following quadratic program:}

\jeffo{
\begin{equation}\label{Payoff QP}
\begin{aligned}
& \underset{x \in \triangle_m}{\max}
& & 2x^TAx\\
& \text{subject to}
& & x^TAx \ge e_i^TAx, \forall i\in \{1,\ldots, m\}.\\
\end{aligned}
\end{equation}
}

\jeffo{One can find an upper bound on this number by solving an LP which searches over symmetric correlated equilibria:

\begin{align}\label{LP1}\tag{LP1}
& \underset{P \in \mathbb{S}^{m, m}}{\max}
& & \Tr(AP^T)& \nonumber\\
& \text{subject to}
& & \sum_{i=1}^m \sum_{j=1}^m P_{i,j} = 1\label{SDP3 Distribution}&\\
&&& \sum_{j=1}^m A_{i,j}P_{i,j} \ge \sum_{j=1}^m A_{k,j}P_{i,j}, \forall i,k \in \{1,\ldots, m\}, & \label{SDP3 CE}\\
&&& P \ge 0.\label{SDP3 Nonnegativity}&
\end{align}

A potentially better upper bound on the maximum welfare can be obtained from a version of \ref{SDP2} adapted to this specific problem: 

\begin{align}\label{SDP3}\tag{SDP3}
& \underset{P \in \mathbb{S}^{m, m}}{\max}
& & \Tr(AP^T)& \nonumber\\
& \text{subject to}
& & (\ref{SDP3 Distribution}), (\ref{SDP3 CE}), (\ref{SDP3 Nonnegativity})\nonumber&\\
&&& P \succeq 0.\nonumber&
\end{align}}

To test the quality of these upper bounds, we tested this \jeffg{LP and SDP} on a random sample of one hundred $5\times 5$ and $10 \times 10$ games\footnote{\jeffg{The matrix $A$ in each game was randomly generated with diagonal entries uniform and independent in [0,.5] and off-diagonal entries uniform and independent in [0,1]}.}. The resulting upper bounds are in Figure~\ref{welfareapprox}, which shows that the bound returned by~\ref{SDP3} was exact in a large number of the experiments.\footnote{The computation of the exact maximum payoffs was done with the \texttt{lrsnash} software~\cite{avis2010enumeration}, which computes all extreme Nash equilibria. For a definition of extreme Nash equilibria and for understanding why it is sufficient for us to compare against extreme Nash equilibria (both in Section \ref{Bounding Payoffs} and in Section~\ref{Strategy Exclusion}), see Appendix~\ref{Lemmas for Extreme Nash Equilibria}. The computation of the SDP upper bound has been implemented in the file {nashbound.m}, which is publicly available at \url{aaa.princeton.edu/software}. \jeffg{This file more generally computes an SDP-based lower bound on the minimum of an input quadratic function over the set of Nash equilibria of a bimatrix game. The file also takes as an argument whether one wishes to only consider symmetric equilibria when the game is symmetric.}}

\jeff{
\begin{figure}[H]
	\includegraphics[height=.25\textheight,keepaspectratio]{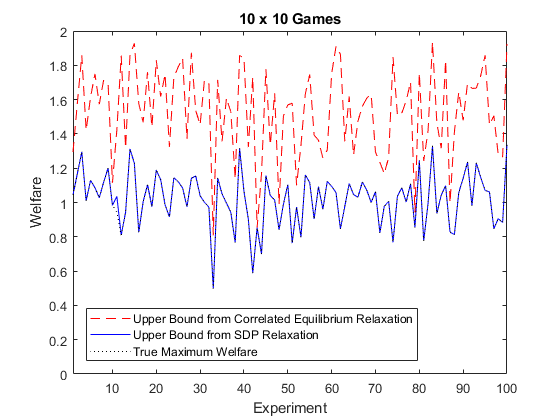}
	\includegraphics[height=.25\textheight,keepaspectratio]{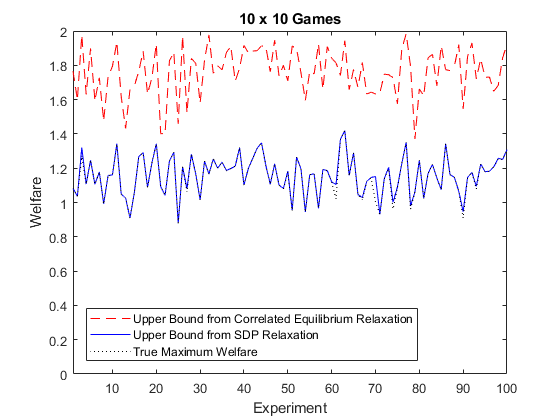}
	\caption{\jeff{The quality of the upper bound on the maximum welfare obtained by\jeffo{~\ref{LP1}} and~\ref{SDP3} on 100 $5\times 5$ games (left) and 100 $10\times 10$ games (right).\label{welfareapprox}}}
\end{figure}
}

\subsection{Strategy Exclusion}\label{Strategy Exclusion}
The strategy exclusion problem asks, given a subset of strategies $\mathcal{S}=(\mathcal{S}_x,\mathcal{S}_y)$, with $ \mathcal{S}_x \subseteq\{1,\ldots,m\}$ and $\mathcal{S}_y\subseteq\{1,\ldots,n\}$, is there a Nash equilibrium in which no strategy in $\mathcal{S}$ is played with positive probability. We will call a set $\mathcal{S}$ ``persistent'' if the answer to this question is negative, i.e. at least one strategy in $\mathcal{S}$ is played with positive probability in every Nash equilibrium. One application of the strategy exclusion problem is to understand whether certain strategies can be discouraged in the design of a game, such as reckless behavior in a game of chicken or defecting in a game of prisoner's dilemma. In these particular examples these strategy sets are persistent and cannot be discouraged.

\jeffo{As in the previous subsection, we consider the strategy exclusion problem for symmetric strategies in symmetric games (such as the aforementioned games of chicken and prisoner's dilemma). A quadratic program which addresses this problem is as follows:

\begin{samepage}
	\begin{equation}\label{Exclusion QP}
	\begin{aligned}
	& \underset{x\in\triangle_m}{\min}
	& & \sum_{i \in \mathcal{S}_x} x_i\\
	& \text{subject to}
	& & x^TAx \ge e_i^TAx, \forall i\in \{1,\ldots, m\}.\\
	\end{aligned}
	\end{equation}
\end{samepage}

Observe that by design, $\mathcal{S}$ is persistent if and only if this quadratic program has a positive optimal value. As in the previous subsection, an LP relaxation of this problem which searches over symmetric correlated equilibria is given by

\begin{samepage}
	\begin{align}\label{LP2}\tag{LP2}
	& \underset{P \in \mathbb{S}^{m, m}}{\min}
	& & \sum_{i \in \mathcal{S}_x}\sum_{j=1}^m P_{ij} &\nonumber\\
	& \text{subject to}
	& & (\ref{SDP3 Distribution}), (\ref{SDP3 CE}), (\ref{SDP3 Nonnegativity}).\nonumber&
	\end{align}
\end{samepage}

The SDP relaxation that we propose for the strategy exclusion problem is the following:

\begin{samepage}
	\begin{align}\label{SDP4}\tag{SDP4}
	& \underset{P \in \mathbb{S}^{m, m}}{\min}
	& & \sum_{i \in \mathcal{S}_x}\sum_{j=1}^m P_{ij} &\nonumber\\
	& \text{subject to}
	& & (\ref{SDP3 Distribution}), (\ref{SDP3 CE}), (\ref{SDP3 Nonnegativity})&\nonumber\\
	&&& P \succeq 0.\nonumber&
	\end{align}
\end{samepage}}

Our approach would be to declare that the strategy set $\mathcal{S}_x$ is persistent if and only if~\ref{SDP4} has a positive optimal value.

Note that since the optimal value of~\ref{SDP4} is a lower bound for that of (\ref{Exclusion QP}),~\ref{SDP4} carries over the property that if a set $\mathcal{S}$ is not persistent, then the SDP for sure returns zero. Thus, when using~\ref{SDP4} on a set which is not persistent, our algorithm will always be correct. However, this is not necessarily the case for a persistent set. While we can be certain that a set is persistent if~\ref{SDP4} returns a positive optimal value (again, because the optimal value of~\ref{SDP4} is a lower bound for that of (\ref{Exclusion QP})), there is still the possibility that for a persistent set~\ref{SDP4} will have optimal value zero. \jeffg{The same arguments hold for the optimal value of \ref{LP2}}.

To test the performance of \jeffg{\ref{LP2} and} \ref{SDP4}, we generated 100 random games of size $5\times 5$ and $10\times 10$ and computed all their \jeffg{symmetric} extreme Nash equilibria\footnote{The exact computation of the exact Nash equilibria was done again with the \texttt{lrsnash} software \cite{avis2010enumeration}, which computes extreme Nash equilibria. To understand why this suffices for our purposes see Appendix~\ref{Lemmas for Extreme Nash Equilibria}.}. We then, for every strategy set $\mathcal{S}$ of cardinality one and two, checked whether that set of strategies was persistent, first by checking among the extreme Nash equilibria, then through \jeffg{\ref{LP2} and} \ref{SDP4}. The results are presented in Tables~\ref{SE Table 5x5 1} and~\ref{SE Table 10x10 1}. As can be seen,~\ref{SDP4} was quite effective for the strategy exclusion problem.

\jeffg{
	\begin{table}[h]
		\caption{\jeffg{Performance of~\ref{LP2} and~\ref{SDP4} on $5\times 5$ games}}
		{\begin{tabular}{|c|c|c|}\hline
				$|\mathcal{S}|$ & 1 & 2\\\hline
				Number of total sets&	500 & 1000\\\hline
				Number of persistent sets&	245 & 748 \\\hline
				Persistent sets certified (\ref{LP2}) & 177 (72.2\%) & 661 (88.7\%)\\\hline
				Persistent sets certified (\ref{SDP4})& 245 (100\%) & 748 (100\%)\\\hline
			\end{tabular}\label{SE Table 5x5 1}}
		\centering
		{}
	\end{table}
	\begin{table}[h]
		\caption{\jeffg{Performance of~\ref{LP2} and~\ref{SDP4} on $10\times 10$ games}}
		{\begin{tabular}{|c|c|c|}\hline
				$|\mathcal{S}|$ & 1 & 2\\\hline
				Number of total sets& 1000 & 4500\\\hline
				Number of persistent sets&	326 & 2383 \\\hline
				Persistent sets certified (\ref{LP2}) & 39 (12.0\%) & 630 (26.4\%)\\\hline
				Persistent sets certified (\ref{SDP4})& 318 (97.5\%) & 2368 (99.4\%)\\\hline
			\end{tabular}\label{SE Table 10x10 1}}
		\centering
		{}
	\end{table}
	}

\section{\jeff{Connection to the Sum of Squares/Lasserre Hierarchy}}\label{Sec: Connection to Sum of Squares/Lasserre Hierarchy}

\jeff{In this section, we clarify the connection of the SDPs we have proposed in this paper to those arising in the sum of squares/Lasserre hierarchy. We start by briefly reviewing this hierarchy.}

\subsection{Sum of Squares/Lasserre Hierarchy}\label{Lassere's Hierarchy}
The sum of squares/Lasserre hierarchy\footnote{The unfamiliar reader is referred to~\cite{lasserre2001global,parrilo2003semidefinite,laurent2009sums} for an introduction to \jeff{this} hierarchy and the related theory of moment relaxations.} gives a recipe for constructing a sequence of SDPs whose optimal values converge to the optimal value of a given polynomial optimization problem. Recall that a \emph{polynomial optimization problem} (pop) is a problem of minimizing a polynomial over a basic semialgebraic set, i.e., a problem of the form
\begin{equation}\label{pop defn}
\begin{aligned}
& \underset{x \in \mathbb{R}^n}{\min}
& & f(x) \\
& \text{subject to}
&& g_i(x) \ge 0, \forall i \in \{1,\ldots,m\},\\
\end{aligned}
\end{equation}
where $f,g_i$ are polynomial functions. 
In this section, when we refer to the $k$-th level of the Lasserre hierarchy, we mean the optimization problem

\begin{equation}\label{Lasserre Hierarchy defn}
\begin{aligned}
\gamma_{sos}^k \defeq & \underset{\gamma,\sigma_i}{\max}
&& \gamma\\
& \text{subject to}
&& f(x)-\gamma = \sigma_0(x) + \sum_{i=1}^m \sigma_i(x)g_i(x),\\
&&&\sigma_i \text{ is sos, }\forall i \in \{0,\ldots,m\},\\
&&&\sigma_0, g_i\sigma_i \text{ have degree at most } 2k,\ \forall i \in \{1,\ldots,m\}.
\end{aligned}
\end{equation}
\jeff{Here, the notation ``sos'' stands for \emph{sum of squares}. We say that a polynomial $p$ is a sum of squares if there exist polynomials $q_1,\ldots,q_r$ such that $p= \sum_{i=1}^r q_i^2$.} There are two primary properties of the Lasserre hierarchy which are of interest. \jeff{The first is that any fixed level of this hierarchy gives an SDP of size polynomial in $n$.} The second is that, if the set $\{x\in\mathbb{R}^n|g_i(x)\ge 0\}$ is Archimedean (see, e.g. \cite{laurent2009sums} for definition), then $\underset{k \to \infty}{\lim} \gamma_{sos}^k = p^*$, where $p^*$ is the optimal value of the pop in (\ref{pop defn}). \jeff{The latter statement is a consequence of Putinar's positivstellensatz~\cite{putinar1993positive},~\cite{lasserre2001global}.}


\subsection{The Lasserre Hierarchy and~\ref{SDP1}}\label{The Lasserre Hierarchy and SDP1}

\jeff{One can show, e.g. via the arguments in \cite{lasserre2009convex}, that the feasible sets of the SDPs dual to the SDPs underlying the hierarchy we summarized above produce an arbitrarily tight outer approximation to the convex hull of the set of Nash equilibria of any game. The downside of this approach, however, is that the higher levels of the hierarchy can get expensive very quickly. This is why the approach we took in this paper was instead to improve the first level of the hierarchy. The next proposition formalizes this connection. 


%
}
\begin{prop}\label{strongerLasserre}
	Consider the problem of minimizing any quadratic objective function over the set of Nash equilibria of a bimatrix game. Then,~\ref{SDP1} (and hence~\ref{SDP2}) gives a lower bound on this problem which is no worse than that produced by the first level of the Lasserre hierarchy.
\end{prop}

\begin{proof}
	
	To prove this proposition we show that the first level of the Lasserre hierarchy is dual to a weakened version of~\ref{SDP1}.
	
	\textbf{Explicit parametrization of first level of the Lasserre hierarchy.} Consider the formulation of the Lasserre hierarchy in (\ref{Lasserre Hierarchy defn}) with $k=1$. Suppose we are minimizing a quadratic function $$f(x,y)=\xy1vec^T \mathcal{C} \xy1vec$$ over the set of Nash equilibria as described by the linear and quadratic constraints in (\ref{QCQP Formulation}). If we apply the first level of the Lasserre hierarchy to this particular pop, we get
	
	\begin{equation}
	\begin{aligned}\label{LH1 long}
	\underset{Q, \alpha, \chi, \beta, \psi, \eta}{\max}
	&& \gamma\\
	\text{subject to}
	&& \xy1vec^T \mathcal{C} \xy1vec-\gamma &= \xy1vec^T Q \xy1vec+\sum_{i=1}^m \alpha_i(x^TAy - e_i^TAy)\\
	&& &+ \sum_{i=1}^n \beta_i(x^TBy - x^TBe_i)\\
	&& &+ \sum_{i=1}^m \chi_i x_i + \sum_{i=1}^n \psi_i y_i\\
	&& &+\eta_1 (\sum_{i=1}^m x_i - 1)+ \eta_2 (\sum_{i=1}^n y_i -1),\\
	&& Q &\succeq 0,\\
	&& \alpha, \chi, \beta, \psi &\ge 0,
	\end{aligned}
	\end{equation}
	where $Q \in \mathbb{S}^{m+n+1 \times m+n+1},\alpha, \chi \in \mathbb{R}^m, \beta, \psi \in \mathbb{R}^n, \eta \in \mathbb{R}^2$.
	
	
	By matching coefficients of the two quadratic functions on the left and right hand sides of (\ref{LH1 long}), this SDP can be written as
	\begin{equation}\label{Lassere Level 1}
	\begin{aligned}
	& \underset{\gamma,\alpha,\beta,\chi,\psi,\eta}{\max}
	&& \gamma\\
	& \text{subject to}
	&& \mathcal{H} \succeq 0,\\
	&&& \alpha, \beta, \chi, \psi \ge 0,
	\end{aligned}
	\end{equation}
	where
	
	\begin{equation}\label{LH1 H}
	\mathcal{H} \mathrel{\mathop:}= \frac{1}{2}\bmat 0 & (-\sum_{i=1}^m \alpha_i )A + (-\sum_{i=1}^m \beta_i)B & \sum_{i=1}^n \beta_i B_{,i}-\chi - \eta_11_m\\
	(-\sum_{i=1}^m \alpha_i)A + (-\sum_{i=1}^n \beta_i)B & 0 & \sum_{i=1}^m \alpha_i A_{i,}^T-\psi - \eta_21_n\\
	\sum_{i=1}^n \beta_iB_{,i}^T-\chi^T-\eta_11_m^T & \sum_{i=1}^m \alpha_i A_{i,}-\psi^T - \eta_21_n^T  & 2\eta_1+2\eta_2-2\gamma
	\emat+\mathcal{C}.\end{equation}

	\textbf{Dual of a weakened version of SDP1.} With this formulation in mind, let us consider a weakened version of~\ref{SDP1} with only the relaxed Nash constraints, unity constraints, and nonnegativity constraints on $x$ and $y$ in the last column (i.e., the nonegativity constraint is not applied to the entire matrix). Let the objective be $\Tr(C\M')$. To write this new SDP in standard form, let\\
	$$\mathcal{A}_i \mathrel{\mathop:}= \frac{1}{2}\bmat 0 & A & 0\\ A^T & 0 & -A_{i,}^T\\ 0 & -A_{i,} & 0\emat, \mathcal{B}_i \mathrel{\mathop:}= \frac{1}{2}\bmat 0 & B & -B_{,i}\\ B^T & 0 & 0\\-B_{,i}^T & 0 & 0 \emat,$$
	$$\mathcal{S}_1 \mathrel{\mathop:}= \frac{1}{2}\bmat 0 & 0 & 1_m\\ 0 & 0 & 0\\1_m^T & 0 & -2\emat, \mathcal{S}_2 \mathrel{\mathop:}= \frac{1}{2}\bmat 0 & 0 & 0\\ 0 & 0 & 1_n\\ 0 & 1_n^T & -2 \emat.$$
	Let $\mathcal{N}_i$ be the matrix with all zeros except a $\frac{1}{2}$ at entry $(i,m+n+1)$ and $(m+n+1,i)$ (or a 1 if $i=m+n+1$).\\
	Then this SDP can be written as\\
	\begin{align}\label{SDP0}\tag{SDP0}
	& \underset{\M'}{\min}
	& & \Tr(\mathcal{C}\M') \\
	& \text{subject to}
	& & \mathcal{M}' \succeq 0,\label{SDP0 PSD}\\
	&&& \Tr(\mathcal{N}_{i}\M') \ge 0, \forall i \in \{1,\ldots,m+n\}, \label{SDP0 nonnegativity}\\
	&&& \Tr(\mathcal{A}_i\M') \ge 0, \forall i \in \{1,\ldots,m\}, \label{SDP0 Nash A}\\
	&&& \Tr(\mathcal{B}_i\M') \ge 0, \forall i \in \{1,\ldots,n\}, \label{SDP0 Nash B}\\
	&&& \Tr(\mathcal{S}_1\M') = 0, \label{SDP0 simplex x}\\
	&&& \Tr(\mathcal{S}_2\M') = 0, \label{SDP0 simplex y}\\
	&&& \Tr({\mathcal{N}}_{m+n+1}) = 1 \label{SDP0 Corner}.
	\end{align}
	
	We now create dual variables for each constraint; we choose $\alpha_i$ and $\beta_i$ for the relaxed Nash constraints (\ref{SDP0 Nash A}) and (\ref{SDP0 Nash B}), $\eta_1$ and $\eta_2$ for the unity constraints (\ref{SDP0 simplex x}) and (\ref{SDP0 simplex y}), $\chi$ for the nonnegativity of $x$ (\ref{SDP0 nonnegativity}), $\psi$ for the nonnegativity of $y$ (\ref{SDP0 nonnegativity}), and $\gamma$ for the final constraint on the corner (\ref{SDP0 Corner}). These variables are chosen to coincide with those used in the parametrization of the first level of the Lasserre hierarchy, as can be seen more clearly below.\\
	
	We then write the dual of the above SDP as\\
	\begin{equation*}
	\begin{aligned}
	& \underset{\alpha, \beta, \lambda, \gamma}{\max}&& \gamma\\
	&\text{subject to} && \sum_{i = 1}^m \alpha_i \mathcal{A}_i + \sum_{i = 1}^n \beta_i \mathcal{B}_i + \sum_{i = 1}^2 \eta_i\mathcal{S}_i+\sum_{i=1}^m \mathcal{N}_{i+n}\chi_i+\sum_{i=1}^n \mathcal{N}_{i}\psi_i + \gamma \mathcal{N}_{m+n+1} \preceq \mathcal{C},\\\
	&&& \alpha, \beta, \chi, \psi \ge 0.
	\end{aligned}
	\end{equation*}
	which can be rewritten as
	\begin{equation} \label{SDP1 Dual}
	\begin{aligned}
	& \underset{\alpha, \beta, \chi, \psi, \gamma}{\max}&& \gamma\\
	&\text{subject to} && \mathcal{G} \succeq 0,\\\
	&&& \alpha, \beta, \chi,\psi \ge 0,
	\end{aligned}
	\end{equation}
	
	where
	\begin{equation*}
	\mathcal{G} \defeq \frac{1}{2}\bmat 0 & (-\sum_{i=1}^m \alpha_i )A + (-\sum_{i=1}^m \beta_i)B & \sum_{i=1}^n \beta_i B_{,i}-\chi - \eta_11_m\\
	(-\sum_{i=1}^m \alpha_i)A + (-\sum_{i=1}^n \beta_i)B & 0 & \sum_{i=1}^m \alpha_i A_{i,}^T-\psi - \eta_21_n\\
	\sum_{i=1}^n \beta_iB_{,i}^T-\chi^T-\eta_11_m^T & \sum_{i=1}^m \alpha_i A_{i,}-\psi^T - \eta_21_n^T  & 2\eta_1+2\eta_2-2\gamma
	\emat+\mathcal{C}.
	\end{equation*}
	
	We can now see that the matrix $\mathcal{G}$ coincides with the matrix $\mathcal{H}$ in the SDP (\ref{Lassere Level 1}). Then we have $$(\ref{LH1 long})^{opt}=(\ref{Lassere Level 1})^{opt} = (\ref{SDP1 Dual})^{opt} \le~\ref{SDP0}^{opt} \le~\ref{SDP1}^{opt},$$
	where the first inequality follows from weak duality, and the second follows from that the constraints of~\ref{SDP0} are a subset of the constraints of~\ref{SDP1}.
\end{proof}

\begin{remark}
	The Lasserre hierarchy can be viewed in each step as a pair of primal-dual SDPs: the sum of squares formulation which we have just presented, and a moment formulation which is dual to the sos formulation~\cite{lasserre2001global}. All our SDPs in this paper can be viewed more directly as an improvement upon the moment formulation. 
\end{remark}

\begin{remark}
	One can see, either by inspection or as an implication of the proof of Theorem~\ref{Necessity of VI}, that in the case where the objective function corresponds to maximizing player A's and/or B's payoffs\footnote{This would be the case, for example, in the maximum social welfare problem of Section~\ref{Bounding Payoffs}, where the matrix of the quadratic form in the objective function is given by $$\mathcal{C}=\bmat 0 & -A-B & 0\\-A-B&0&0\\0&0&0\emat.$$}, SDPs (\ref{Lassere Level 1}) and (\ref{SDP1 Dual}) are infeasible. This means that for such problems the first level of the Lasserre hierarchy gives an upper bound of $+\infty$ on the maximum payoff. On the other hand, the additional valid inequalities in~\ref{SDP2} guarantee that the resulting bound is always finite.
\end{remark}

\section{Future Work} \label{Conclusion}

Our work leaves many avenues of further research. Are there other interesting subclasses of games (besides \jeff{strictly competitive games}) for which our SDP is guaranteed to recover an exact Nash equilibrium? Can the guarantees on $\epsilon$ in Section~\ref{Bounds} be improved in the rank-2 case (or the general case) by improving our analysis? Is there a polynomial time algorithm that is guaranteed to find a rank-2 solution to~\ref{SDP2}? Such an algorithm, together with our analysis, would improve the best known approximation bound for symmetric games (see Theorem \ref{1/3se}). \jeffb{Can this bound be extended to general games? We show in Appendix \ref{Sec: Irreducibility} that some natural approaches based on symmetrization of games do not immediately lead to a positive answer to this question.} Can SDPs in a higher level of the Lasserre hierarchy be used to achieve better $\epsilon$ guarantees? What are systematic ways of adding valid inequalities to these higher-order SDPs by exploiting the structure of the Nash equilibrium problem? \jeff{For example, since any strategy played with positive probability must give the same payoff, one can add a relaxed version of the cubic constraints $$x_ix_j(e_i^TAy - e_j^TAy)=0, \forall i,j \in \{1,\ldots,m\}$$ to the SDP underlying the second level of the Lasserre hierarchy. What are other valid inequalities for the second level?} Finally, our algorithms were specifically designed for two-player one-shot games. This leaves open the design and analysis of semidefinite relaxations for repeated games or games with more than two players.

{\aaa 
\section*{Acknowledgments.}
We would like to thank Ilan Adler, Costis Daskalakis, Georgina Hall, Ramon van Handel, and Robert Vanderbei for insightful exchanges. We are also extremely grateful to an anonymous referee for various insightful questions and comments which have led to a significantly improved version of this manuscript.
}

\bibliographystyle{abbrv}
\bibliography{Nash_SDP_refs}

\appendix

\section{Statistics on $\epsilon$ from Algorithms in Section~\ref{Algorithms}}\label{Appendex Statistics}
Below are statistics for the $\epsilon$ recovered in 100 random games of varying sizes using the algorithms of Section~\ref{Algorithms}. 

\begin{table}[H]
	\caption{Statistics on $\epsilon$ for $5\times 5$ games after 20 iterations.\label{5x5}}
	{\begin{tabular}{|c|c|c|c|c|c|}\hline
			Algorithm & Max & Mean & Median & StDev\\\hline
			Square Root&	0.0702&	0.0040&	0.0004&	0.0099	\\\hline
			Diagonal Gap&	0.0448&	0.0027&	0&	0.0061	\\\hline
	\end{tabular}}
	\centering
	{}
\end{table}

\begin{table}[H]
	\caption{Statistics on $\epsilon$ for $10\times 5$ games after 20 iterations.\label{10x5}}
	{\begin{tabular}{|c|c|c|c|c|c|}\hline
			Algorithm & Max & Mean & Median & StDev\\\hline
			Square Root&	0.0327&	0.0044&	0.0021&	0.0064	\\\hline
			Diagonal Gap&	0.0267&	0.0033&	0.0006&	0.0053	\\\hline
	\end{tabular}}
	\centering
	{}
\end{table}

\begin{table}[H]
	\caption{Statistics on $\epsilon$ for $10\times 10$ games after 20 iterations.\label{10x10}}
	{\begin{tabular}{|c|c|c|c|c|c|}\hline
			Algorithm & Max & Mean & Median & StDev\\\hline
			Square Root&	0.0373&	0.0058&	0.0039&	0.0065	\\\hline
			Diagonal Gap&	0.0266&	0.0043&	0.0026&	0.0051	\\\hline
	\end{tabular}}
	\centering
	{}
\end{table}

\begin{table}[H]
	\caption{Statistics on $\epsilon$ for $15\times 10$ games after 20 iterations.\label{15x10}}
	{\begin{tabular}{|c|c|c|c|c|c|}\hline
			Algorithm & Max & Mean & Median & StDev\\\hline
			Square Root&	0.0206&	0.0050&	0.0034&	0.0045	\\\hline
			Diagonal Gap&	0.0212&	0.0038&	0.0025&	0.0039	\\\hline
	\end{tabular}}
	\centering
	{}
\end{table}

\begin{table}[H]
	\caption{Statistics on $\epsilon$ for $15\times 15$ games after 20 iterations.\label{15x15}}
	{\begin{tabular}{|c|c|c|c|c|c|}\hline
			Algorithm & Max & Mean & Median & StDev\\\hline
			Square Root&	0.0169&	0.0051&	0.0042&	0.0039	\\\hline
			Diagonal Gap&	0.0159&	0.0038&	0.0029&	0.0034	\\\hline
	\end{tabular}}
	\centering
	{}
\end{table}

\begin{table}[H]
	\caption{Statistics on $\epsilon$ for $20\times 15$ games after 20 iterations.\label{15x20}}
	{\begin{tabular}{|c|c|c|c|c|c|}\hline
			Algorithm & Max & Mean & Median & StDev\\\hline
			Square Root&	0.0152&	0.0046&	0.0035&	0.0036	\\\hline
			Diagonal Gap&	0.0119&	0.0032&	0.0022&	0.0027	\\\hline
	\end{tabular}}
	\centering
	{}
\end{table}

\begin{table}[H]
	\caption{Statistics on $\epsilon$ for $20\times 20$ games after 20 iterations.\label{20x20}}
	{\begin{tabular}{|c|c|c|c|c|c|}\hline
			Algorithm & Max & Mean & Median & StDev\\\hline
			Square Root&	0.0198&	0.0046&	0.0039&	0.0034	\\\hline
			Diagonal Gap&	0.0159&	0.0032&	0.0024&	0.0032	\\\hline
	\end{tabular}}
	\centering
	{}
\end{table}

\jeffg{\section{Comparison with an SDP Approach from \cite{laraki2012semidefinite}}\label{Sec: Comparison}
In this section, at the request of a referee, we compare the first level of the SDP hierarchy given in \cite[Section 4]{laraki2012semidefinite} to~\ref{SDP2} using $\Tr(M)$ as the objective function on 100 randomly generated games for each size given in the tables below. The first level of the hierarchy in \cite{laraki2012semidefinite} optimizes over a matrix which is slightly bigger than the one in \ref{SDP2}, though it has a number of constraints linear in the size of the game considered, as opposed to the quadratic number in \ref{SDP2}. We remark that the approach in \cite{laraki2012semidefinite} is applicable more generally to many other problems, including several in game theory.

The scalar $\epsilon$ reported in Table \ref{Table: SDP5 Statistics} is computed using the strategies $(x,y)$ extracted from the first row of the optimal matrix $M_1$ as described in Section 4.1 of \cite{laraki2012semidefinite}. The scalar $\epsilon$ reported in Table \ref{Table: Trace Statistics} is computed using $x = P1_n$ and $y = P^T1_m$ from the optimal solution to \ref{SDP2} with $\Tr(\M)$ as the objective function.

\begin{table}[H]
	\caption{Statistics on $\epsilon$ for first level of the hierarchy in \cite{laraki2012semidefinite}.\label{Table: SDP5 Statistics}}
	{\begin{tabular}{|c|c|c|c|c|c|c|c|c|c|c|}\hline
			& $5 \times 5$& $10 \times 5$ & $10 \times 10$ & $15 \times 10$ & $15 \times 15$ & $20 \times 15$ & $20 \times 20$ \\\hline
			Max &	0.3357 &	0.3304 &	0.2557 &	0.2189 &	0.1987 &	0.1837 &	0.1828 \\\hline
			Mean &	0.1883 &	0.1889 &	0.1513 &	0.1446 &	0.1262 &	0.1217 &	0.1087 \\\hline
			Median &	0.1803 &	0.1865 &	0.1452 &	0.1418 &	0.1271 &	0.1208 &	0.1070 \\\hline
	\end{tabular}}
	\centering
\end{table}

\begin{table}[H]
	\caption{Statistics on $\epsilon$ for~\ref{SDP2} with $\Tr(M)$ as the objective function.\label{Table: Trace Statistics}}
	{\begin{tabular}{|c|c|c|c|c|c|c|c|c|c|c|}\hline
			& $5 \times 5$ & $10 \times 5$ & $10 \times 10$ & $15 \times 10$ & $15 \times 15$ & $20 \times 15$ & $20 \times 20$ \\\hline
			Max &	0.1581 &	0.1589 &	0.115 &	0.1335 &	0.0878 &	0.082 &	0.0619	\\\hline
			Mean &	0.0219 &	0.0332 &	0.0405 &	0.04 &	0.0366 &	0.0356 &	0.0298 	\\\hline
			Median &	0.0046 &	0.0233 &	0.036 &	0.0346 &	0.0345 &	0.0325 &	0.0293 	\\\hline
	\end{tabular}}
	\centering
\end{table}

%
%
We also ran the second level of the hierarchy in \cite{laraki2012semidefinite} on the same 100 $5 \times 5$ games. The maximum $\epsilon$ observed was .3362, while the mean was .1880 and the median was .1800. The size of the variable matrix that needs to be positive semidefinite for this level is $78 \times 78$.


}

\section{Lemmas for Extreme Nash Equilibria}\label{Lemmas for Extreme Nash Equilibria}
The results reported in Section~\ref{Bounding Payoffs and Strategy Exclusion} were found using the \texttt{lrsnash}~\cite{avis2010enumeration} software which computes extreme Nash equilibria (see definition below). In particular the true maximum welfare and the persistent strategy sets were found in relation to extreme \jeffg{symmetric} Nash equilibria only. We show in this appendix why this is sufficient for the claims we made about \emph{all} \jeffg{symmetric} Nash equilibria. \jeffo{We prove a more general statement below about general games and general Nash equilibria since this could be of potential independent interest. The proof for symmetric games is identical once the strategies considered are restricted to be symmetric.}
\begin{defn}
	An \emph{extreme Nash equilibrium} is a Nash equilibrium which cannot be expressed as a convex combination of other Nash equilibria.
\end{defn}

\begin{lem}\label{lem: convex combination ENE}
	All Nash equilibria are convex combinations of extreme Nash equilibria.
\end{lem}
\begin{proof}
	It suffices to show that any extreme point of the convex hull of the set of Nash equilibria must be an extreme Nash equilibrium, as any point in a compact convex set can be written as a convex combination of its extreme points. Note that this convex hull contains three types of points: extreme Nash equilibria, Nash equilibria which are not extreme, and convex combinations of Nash equilibria which are not Nash equilibria. The claim then follows because any extreme point of the convex hull cannot be of the second or third type, as they can be written as convex combinations of other points in the hull.
\end{proof}

The next lemma shows that checking extreme Nash equilibria are sufficient for the maximum welfare problem.

\begin{lem}
	For any bimatrix game, there exists an extreme Nash equilibrium giving the maximum welfare among all Nash equilibria.
\end{lem}
\begin{proof}
	
	Consider any Nash equilibrium $(\tilde{x}, \tilde{y})$, and let it be written as $\bmat \tilde{x} \\ \tilde{y} \emat = \sum_{i=1}^r \lambda_i \bmat x^i \\ y^i \emat$ for some set of extreme Nash equilibria $\bmat x^1 \\ y^1 \emat, \ldots, \bmat x^r \\ y^r \emat$ and $\lambda \in \triangle_r$. Observe that for any $i,j$, \begin{equation}\label{eq: defn of NE}x^{iT}Ay^{j} \le x^{jT}Ay^{j}, x^{iT}By^{j} \le x^{iT}By^{i},\end{equation} from the definition of a Nash equilibrium. Now note that
	
	\begin{equation*}
	\begin{aligned}
	\tilde{x}^T(A+B)\tilde{y} & = (\sum_{i=1}^r \lambda_i x^i)^T(A+B)(\sum_{i=1}^r \lambda_i y^i)\\
	&=\sum_{i=1}^r \sum_{j=1}^r \lambda_i\lambda_j x^{iT}(A+B)y^j\\
	&= \sum_{i=1}^r \sum_{j=1}^r \lambda_i\lambda_j x^{iT}Ay^j + \sum_{i=1}^r \sum_{j=1}^r \lambda_i\lambda_j x^{iT}By^j\\
	& \overset{(\ref{eq: defn of NE})}{\le} \sum_{i=1}^r \sum_{j=1}^r \lambda_i\lambda_j x^{jT}Ay^j + \sum_{i=1}^r \sum_{j=1}^r \lambda_i\lambda_j x^{iT}By^i\\
	&= \sum_{i=1}^r \lambda_i x^{iT}Ay^i + \sum_{i=1}^r \lambda_i x^{iT}By^i\\
	&= \sum_{i=1}^r \lambda_i x^{iT}(A+B)y^i.
	\end{aligned}
	\end{equation*}
	
	In particular, since each $(x^i, y^i)$ is an extreme Nash equilibrium, this tells us for any Nash equilibrium $(\tilde{x}, \tilde{y})$ there must be an extreme Nash equilibrium which has at least as much welfare.
\end{proof}

Similarly for the results for persistent sets in Section~\ref{Strategy Exclusion}, there is no loss in restricting attention to extreme Nash equilibria.
\begin{lem}
	For a given strategy set $\mathcal{S}$, if every extreme Nash equilibrium plays at least one strategy in $\mathcal{S}$ with positive probability, then every Nash equilibrium plays at least one strategy in $\mathcal{S}$ with positive probability.
\end{lem}

\begin{proof}
	Let $\mathcal{S}$ be a persistent set of strategies. Since all Nash equilibria are composed of nonnegative entries, and every extreme Nash equilibrium has positive probability on some entry in $\mathcal{S}$, any convex combination of extreme Nash equilibria must have positive probability on some entry in $\mathcal{S}$.
	
\end{proof}

\jeffo{
\section{A Note on Reductions from General Games to Symmetric Games}\label{Sec: Irreducibility}

An anonymous referee asked us if our guarantees for symmetric games transfer over to general games by symmetrization. Indeed, there are reductions in the literature that take a general game, construct a symmetric game from it, and relate the Nash equilibria of the original game to symmetric Nash equilibria of its symmetrized version. In this Appendix, we review two well-known reductions of this type~\cite{griesmer1963symmetric,jurg1992symmetrization} and show that the quality of approximate Nash equilibria can differ greatly between the two games. We hope that our examples can be of independent interest.


\subsection{The Reduction of Griesmer, Hoffman, and Robinson \cite{griesmer1963symmetric}}

Consider a game $(A,B)$ with $A, B > 0$ and a Nash equilibrium $(x^*, y^*)$ of it with payoffs $p_A~\defeq~x^{*T}Ay^*$ and $p_B~\defeq~x^{*T}By^*$. Then the symmetric game $(S_{AB}, S_{AB}^T)$ with

$$S_{AB} \defeq \bmat
0 & A\\
B^T & 0
\emat$$
admits a symmetric Nash equilibrium in which both players play $\bmat \frac{p_A}{p_A+p_B}x^* \\ \frac{p_B}{p_A+p_B}y^*\emat$. In the reverse direction, any symmetric equilibrium $\left( \bmat x\\ y\emat, \bmat x\\ y\emat\right)$ of $(S_{AB}, S_{AB}^T)$ yields a Nash equilibrium $(\frac{x}{1_m^Tx}, \frac{y}{1_n^Ty})$ to the original game $(A,B)$.

To demonstrate that high-quality approximate Nash equilibria in the symmetrized game can map to low-quality approximate Nash equilibria in the original game, consider the game given by $(A,B) = \left( \bmat \epsilon & 0 \\ 1 & 1\emat, \bmat \epsilon^2 & 0 \\ 0 & 1 \emat \right)$ for some $\epsilon > 0$. The symmetric strategy $$\left( \bmat \frac{1}{1 + \epsilon} \\ 0 \\ \frac{\epsilon}{1+\epsilon} \\ 0 \emat, \bmat \frac{1}{1 + \epsilon} \\ 0 \\ \frac{\epsilon}{1+\epsilon} \\ 0 \emat \right)$$ is an $\epsilon \frac{1-\epsilon}{1+\epsilon}$-NE for $(S_{AB}, S_{AB}^T)$, but the strategy pair $\left(\bmat 1 \\ 0 \emat, \bmat 1 \\ 0 \emat\right)$ is a $(1-\epsilon)$-NE for $(A,B)$.

\subsection{The Reduction of Jurg, Jansen, Potters, and Tijs \cite{jurg1992symmetrization}}
Consider a game $(A,B)$ with $A > 0, B < 0$ and a Nash equilibrium $(x^*, y^*)$ of it with payoffs $p_A~\defeq~x^{*T}Ay^*$ and $p_B~\defeq~x^{*T}By^*$. Then the symmetric game $(S_{AB}, S_{AB}^T)$ with

$$S_{AB} \defeq \bmat
0_{m \times m} & A & -1_m\\
B^T & 0_{n \times n} & 1_n\\
1_m^T & -1_n^T & 0\\
\emat$$
admits a symmetric Nash equilibrium in which both players play $$\bmat \frac{x^*}{2 - p_B} \\ \frac{y^*}{2+p_A} \\ 1 - \frac{1}{2-p_B} - \frac{1}{2+p_A}\emat.$$ In the reverse direction, any symmetric equilibrium $\left( \bmat x\\ y\\z \emat, \bmat x\\ y\\z\emat\right)$ of $(S_{AB}, S_{AB}^T)$ yields a Nash equilibrium $(\frac{x}{1_m^Tx}, \frac{y}{1_n^Ty})$ to the original game $(A,B)$. This reduction has some advantages over the previous one (see \cite[Section 1]{jurg1992symmetrization}).

To demonstrate that high-quality approximate Nash equilibria in the new symmetrized game can again map to low-quality approximate Nash equilibria in the original game, consider the game given by $(A,B) = \left( \bmat 0 & 0 \\ 0 & 1\emat, \bmat -1 & -1 \\ 0 & 0 \emat \right)$. Let $\epsilon \in (0,\frac{1}{2})$. The symmetric strategy $$\left( \bmat \epsilon \\ 0 \\ 1 - \epsilon \\ 0 \\ 0\emat, \bmat \epsilon \\ 0 \\ 1 - \epsilon \\ 0 \\ 0\emat \right)$$ is an $\frac{\epsilon}{2}(1-\epsilon)$-NE\footnote{Note that approximation factor is halved since the range of the entries of the payoff matrix in the symmetrized game is $[-1,1]$.} for $(S_{AB}, S_{AB}^T)$, but the strategy pair $\left(\bmat 1 \\ 0 \emat, \bmat 1 \\ 0 \emat\right)$ is a $1$-NE for $(A,B)$.}



\end{document}